\documentclass[12pt]{amsart}
\usepackage{geometry}  
\usepackage{graphicx}
\usepackage{amssymb}
\usepackage{amsmath}
\usepackage{color}
\usepackage{hyperref}
\usepackage{pdfsync}
 \usepackage{tikz}
 \usepackage{tikz-cd}

\usepackage[all]{xy}
\xyoption{matrix}
\xyoption{arrow}

\usetikzlibrary{arrows,decorations.pathmorphing,backgrounds,positioning,fit,petri}

\newtheorem{Theorem}{Theorem}[section]
\newtheorem{Lemma}[Theorem]{Lemma}
\newtheorem{Corollary}[Theorem]{Corollary}
\newtheorem{Proposition}[Theorem]{Proposition}

\theoremstyle{definition}
\newtheorem{Definition}[Theorem]{Definition}
\newtheorem{Example}[Theorem]{Example}

\newtheorem{Remark}[Theorem]{Remark}
 
\numberwithin{equation}{section}

\newcommand{\op}{^{op}}

\newcommand{\mat}[1]{\ensuremath{
\left[\begin{matrix}#1
\end{matrix}\right]
}}
 
\newcommand{\vs}[1]{\vskip .#1 cm} 
\newcommand{\then}{\Rightarrow}

\newcommand{\ifff}{\Leftrightarrow}

 \newcommand{\onto}{\twoheadrightarrow}

\newcommand{\ot}{\leftarrow}

\DeclareMathOperator{\Tr}{Tr}
\DeclareMathOperator{\Gen}{Gen}%
\DeclareMathOperator{\Hom}{Hom}%
\DeclareMathOperator{\Ext}{Ext}%
\DeclareMathOperator{\End}{End}%
\DeclareMathOperator{\soc}{soc} 
\DeclareMathOperator{\topp}{top} 

\DeclareMathOperator{\undim}{\underline{dim}}
 
\newcommand{\field}[1]{\mathbb{#1}}

\newcommand{\CC}{\ensuremath{{\field{C}}}}
\newcommand{\RR}{\ensuremath{{\field{R}}}}

\newcommand{\QQ}{\ensuremath{{\field{Q}}}}

\newcommand{\symm}[2]{{\dfrac{\bg{#1}}{\color{red}#2}}}
\newcommand{\tsymm}[2]{{\frac{\bg{#1}}{\color{red}#2}}}

\newcommand{\bg}[1]{\color{black!45!green}#1}

\newcommand{\commentout}[1]{}

\newcommand{\cB}{\ensuremath{{\mathcal{B}}}}
\newcommand{\cC}{\ensuremath{{\mathcal{C}}}}
\newcommand{\cD}{\ensuremath{{\mathcal{D}}}}

\newcommand{\cG}{\ensuremath{{\mathcal{G}}}}

\newcommand{\cT}{\ensuremath{{\mathcal{T}}}}
\newcommand{\cU}{\ensuremath{{\mathcal{U}}}}

\newcommand{\cW}{\ensuremath{{\mathcal{W}}}}

\newcommand{\no}[1]{}

\title{Balanced notation for $\tau$-rigid pairs}
\author{Kiyoshi Igusa}
\address{Department of Mathematics, Brandeis University, Waltham, MA 02454}\email{igusa@brandeis.edu}
\author{Gordana Todorov}
\address{Department of Mathematics, Northeastern University, Boston, MA 02115}\email{g.todorov@northeastern.edu}



\begin{document}

\begin{abstract}
{We introduce a new notation for $\tau$-rigid pairs called ``balanced pairs''. The notation $\tsymm AB$ allows for simplification of some formulas, for example the Jasso category and its dual. Using this balanced notation we show that the $g$-vector $g(\tsymm AB)$ has nice geometric and algebraic implications: It defines ``lower'' and ``upper'' chambers and we prove that this corresponds to generalized Bongartz and co-Bongartz completions of balanced pairs. We also show that the balanced notation agrees the wall labels for the semi-invariant picture in the hereditary case and, in the general case, we describe the relation between balanced notation and the wall labels.}
\end{abstract}

\maketitle



{\section*{Introduction} 

{
We introduce ``balanced'' notation for $\tau$-rigid pairs (Def. \ref{def: balanced notation A/B}) in order to have balanced use of projective and injective modules. This led to analysis of two important chambers $\cB_-$ (Def. \ref{def: chamber A}) and $\cB_+$ (Def. \ref{def: chamber B}) associated to balanced pairs. Before stating the theorems, here is an example for the hereditary path algebra of the $A_3$-quiver. Figure \ref{Fig01: A3 example} is a ``semi-invariant picture'' of the domains of semi-invariants for this hereditary algebra. For general finite dimensional algebras, these are called ``stability diagrams''. In the figure, vertices are $g$-vectors of indecomposable modules or shifted projective modules which we want to treat in the same way. 

{%
\begin{figure}[htbp]
\begin{center}
\begin{tikzpicture}[scale=.9]

\draw (8,-1) node{$\xymatrixrowsep{10pt}\xymatrixcolsep{10pt}
\xymatrix{
&& P_3\ar[dr]\ar@{--}[rr] && P_1[1]\ar[dr]\\
& P_2\ar[ur]\ar[dr] \ar@{--}[rr] && I_2\ar[dr]\ar@{--}[rr]\ar[ur]&&P_2[1]\ar[dr]\\
S_1\ar[ur] \ar@{--}[rr]&&S_2\ar[ur] \ar@{--}[rr]&&S_3\ar@{--}[rr]\ar[ur]&&P_3[1]
	}$};

\begin{scope}
\coordinate (A) at (-1.25,-2);
\coordinate (B) at (1.95,-.5);
\draw[thick] (0,0) ellipse[x radius= 2cm,y radius=1.73cm];
\draw[thick][fill,white] (-1,0) circle[radius=2cm];
\draw[thick] (-1,0) circle[radius=2cm];
\draw[thick] (1,0) circle[radius=2cm];
\draw[thick] (0,-2) ellipse[x radius=3cm,y radius=2cm];
\draw[thick] (A)..controls (1,-3.5) and (2,-2)..(B);
\coordinate (D1) at (2.6,-3.5);
\draw (D1) node{\tiny$\cD_H(S_3)$};
\draw (2.8,-3.9) node{\tiny$=\cD_H(I_3)$};
\coordinate (D2) at (-3,1.3);
\coordinate (D2r) at (3.5,.5);
\draw (D2) node{\tiny$\cD_H(S_1)$};
\draw (D2r) node{\tiny$\cD_H(S_2)$};

\coordinate (D3) at (0,-3.5);
\draw (D3) node{\tiny$\cD_H(I_2)$};

\coordinate (D4) at (1,-2.8);
\coordinate (D4u) at (1.9,1.2);
\draw (D4) node{\tiny$\cD_H(P_3)$};
\draw (D4u) node{\tiny$\cD_H(P_2)$};

\draw[black!45!green](0,-2) node{$P_3$};
\draw[black!45!green](0,2.1) node{$P_3[1]$};
\draw[black!45!green](1.25,0.05) node{$P_2$};
\draw[black!45!green](-1.25,0.08) node{$S_1$};
\draw[black!45!green](3.2,-1.1) node{$P_1[1]$};

\draw[black!45!green](-3.2,-1.1) node{$P_2[1]$};

\draw[black!45!green](-1.45,-2.2) node{$S_3$};

\draw[black!45!green](2.2,-.3) node{$S_2$};
\draw[black!45!green,<-] (2.4,-.15) -- (3,.1);

\draw[black!45!green](1.8,-2.2) node{$I_2$};


\clip (0,-2) ellipse[x radius=3cm,y radius=2cm];
\draw[thick] (.65,-1.3) circle[radius=2.04cm];

\end{scope}
\end{tikzpicture}
\caption{On the left is the semi-invariant picture for the quiver $Q:1\ot 2\ot 3$ with semi-invariant domains labeled as $\cD_H(M)$ for the six indecomposable modules $M$. The vertices are labeled in the standard way by the nine objects shown in green. The Auslander-Reiten quiver of $Q$ together with the shifted projectives is on the right. 
}
\label{Fig01: A3 example}
\end{center}
\end{figure}
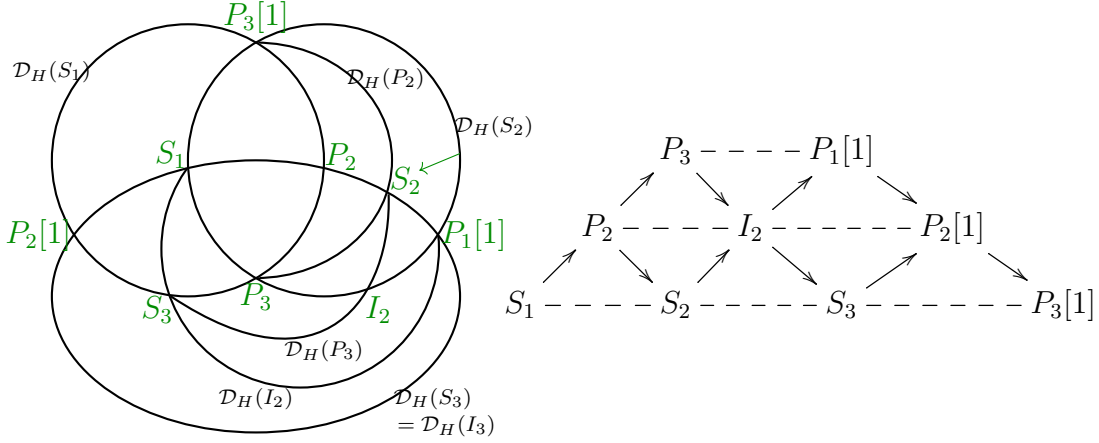
}

The results of this paper summarized in the introduction are for finite dimensional algebras over any field except for Theorem \ref{thm 2} which is for hereditary algebras.

\begin{Theorem}[Corollaries \ref{cor: vertices of U+}, \ref{cor: vertices of U0}]\label{thm 1}
Let $\Lambda$ be a finite dimensional algebra. The $g$-vector of any non-complete balanced pair $\symm AB$ has a unique geometrically defined upper chamber $\cB_-$ and a unique lower chamber $\cB_+$.
\end{Theorem}

Upper and lower chambers are defined, or rather, uniquely characterized in Def. \ref{def: chamber A} and \ref{def: chamber B}. They are eventually shown to exist in Corollaries \ref{cor: vertices of U+} and \ref{cor: vertices of U0} and illustrated in this introduction in Figures \ref{Fig02: A3 example} and \ref{Fig: almost complete A/B}.

{\begin{center}
    \begin{figure}
\begin{tikzpicture}[scale=.7]
\draw[ thick] (0,0)..controls (1,1) and (3,2)..(4,2);
\draw[ thick] (0,0)..controls (-1,1) and (-3,2)..(-4,2);
\draw[ thick] (0,0)..controls (-1,-1) and (-2,-2)..(-2.5,-4);
\draw[ thick] (0,0)..controls (1,-1) and (2,-2)..(2.5,-4);
\draw[thick] (-4,1)..controls (-3,2) and (-1.4,3)..(0,3);
\draw[thick] (4,1)..controls (3,2) and (1.4,3)..(0,3);

\begin{scope}[yshift=-.5mm]
\clip (-3,1.3)rectangle(3,3.1);
\draw[ thick,red] (-4,1)..controls (-3,2) and (-1.4,3)..(0,3);
\draw[thick,red] (4,1)..controls (3,2) and (1.4,3)..(0,3);
\end{scope}
\draw[thick] (-3.5,-4)..controls (-2.5,-3.5) and (-1.5,-3).. (0,-3); 
\draw[thick] (3.5,-4)..controls (2.5,-3.5) and (1.5,-3).. (0,-3); 

\begin{scope}[yshift=.5mm]
\clip (-2.3,-4)rectangle (2.3,-2.7);
\draw[thick,black!45!green] (-3.5,-4)..controls (-2.5,-3.5) and (-1.5,-3).. (0,-3); 
\draw[thick,black!45!green] (3.5,-4)..controls (2.5,-3.5) and (1.5,-3).. (0,-3); 
\end{scope}
\draw[thick] (-4,0)..controls (-3,.5) and (-1,.3)..(0,0);
\draw[thick] (-4,-1)..controls (-3,-.5) and (-1,0)..(0,0);
\draw[thick] (-4,-1.5)..controls (-3,-1) and (-1,-.1)..(0,0);
\draw[thick] (-4,-2)..controls (-3,-1.5) and (-1,-.2)..(0,0);
\draw (0.1,-.05)node[right]{$v_0$};
\draw (0,1.5) node{$\cB_-$};
\draw (0,-1.7) node{$\cB_+$};
%
\draw[black!45!green] (0,.5) node{$A$};
\draw[red](0,2.5) node{$\cD_H(A)$};
\draw[red] (0,-.5) node{$B$};
\draw[black!45!green] (-1.2,-2.7) node[right]{$\cD_H(B)$};
\end{tikzpicture}
\caption{The vertex $v_0$ is the $g$-vector of $A$ which labels the ``covering wall'' $\cD_H(A)$. ``Covering'' means the ``upper chamber'' $\cB_-$ is on the positive (red) side of $\cD_H(A)$ and on the negative (green) side of all of its other walls. On the other side, the ``lower chamber'' $\cB_+$ is on the negative (green) side of $\cD_H(B)$ and on the positive (red) side of all of its other walls. Instead of just $A$ we would like to label the vertex $v_0$ with the pair of modules $\symm AB$ as shown. We also want the label $A$ in the region $\cB_-$ covered by $\cD_H(A)$ and the label $B$ in the region $\cB_+$ on the negative side of $\cD_H(B)$.}
\label{Fig02: A3 example}
\end{figure}
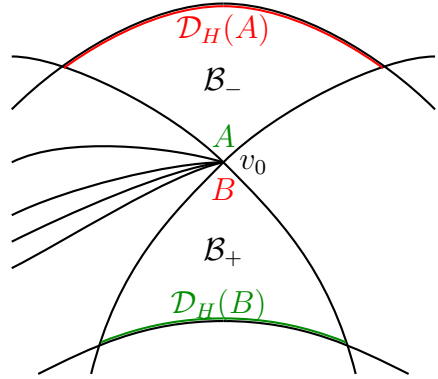
\end{center}
}

{
\begin{figure}[htbp]
\begin{center}
\begin{tikzpicture}[scale=.8]

\begin{scope}[xshift=-5mm]
\draw[very thick] (0,0)..controls (1,0) and (2,-.2)..(3,-.5);
\draw[very thick] (0,0)..controls (-1,0) and (-2,-.2)..(-3,-.5);
\end{scope}
\draw[very thick] (5,3.6)--(4.5,3)--(2,0)--(1,-1.2);
\draw[very thick] (5.8,3.6)--(4.5,3)--(-2,0)--(-4.6,-1.2);
\begin{scope}[xshift=2mm]
\draw[very thick] (-6.8,-3.1)--(-5,-2.5)--(4,.5);
\draw[very thick] (-5.6,-3.1)--(-5,-2.5)--(-2,.5);
\end{scope}
\draw (2,-.5) node[below]{$v_{n-1}$};
\draw (-2.4,-.3) node[below]{$v_1$};
\draw (4.6,2.9) node[right]{$v_0=g(\tsymm{A_0}{B_0})$};
\draw (-4.9,-2.7) node[right]{$v_0'=g(\tsymm{A_0'}{B_0'})$};
\draw (1.5,.7) node{$\cB_-$};
\draw (-1.5,-.9) node{$\cB_+$};
\draw (0,0) node{$\bullet$};
\draw (0,0)node[below]{$\sum v_i$};
\end{tikzpicture}
\caption{For an almost complete balanced pair $\symm AB=\bigsqcup \symm{A_i}{B_i}$, the upper and lower chambers $\cB_-$ and $\cB_+$ are the only two chambers which contain $g\!\left(\symm AB\right)=\sum v_i$ where $v_i=g\!\left(\symm {A_i}{B_i}\right)$ on their boundaries. This is because walls have only two sides. The remaining vertices $v_0$, $v_0'$ are the $g$-vectors of the two ways to complete the almost balanced pair.}
\label{Fig: almost complete A/B}
\end{center}
\end{figure}
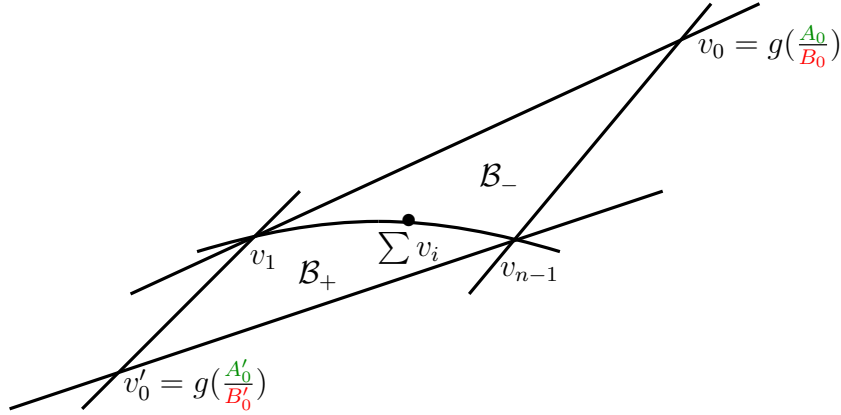
}

\begin{Theorem}[Theorems \ref{thm: P(U0),P(U+)}, \ref{thm: Bongartz complement}]\label{thm 1a}
Let $\Lambda$ be a f.d. algebra. The torsion class associated to $\cB_+$ is $\,^\perp B$ and, therefore, the lower chamber $\cB_+$ of $g\!\left(\symm AB\right)$ corresponds to the generalized Bongartz completion of $\symm AB$ (Def. \ref{def: generalized Bongartz}), i.e., the other vertices of the lower chamber $\cB_+$ give the generalized Bongartz complement of $\symm AB$.
\end{Theorem}

The chamber $\cB_+$ corresponds to the actual Bongartz completion of $A$ when $\,^\perp B$ is sincere (and, by Remark \ref{rem: A'' are nonzero}, this happens if and only if every component of $A$ is nonzero, i.e., $\tsymm 0{I_j}$ is not a component of $\tsymm AB$).

\begin{Theorem}[Theorem \ref{thm: P(U0),P(U+)}]\label{thm 1b}
Let $\Lambda$ be a f.d. algebra. The torsion class associated to $\cB_-$ is $\Gen A$ and, therefore, the upper chamber $\cB_-$ of $g\!\left(\symm AB\right)$ corresponds to the generalized co-Bongartz completion of $\symm AB$ (Def. \ref{def: generalized coB}), i.e., the other vertices of the upper chamber $\cB_-$ give the generalized co-Bongartz complement of $\symm AB$.
\end{Theorem}

It is known that chambers are uniquely determined by their associated torsion classes (Theorem \ref{thm: chambers are convex}). Thus, the upper chamber $\cB_-$ is uniquely determined by $A$ and the lower chamber $\cB_+$ is uniquely determined by $B$. In Remark \ref{rem: too many torsion classes} we give simple examples where not all torsion classes come from chambers.

As an example of Theorems \ref{thm 1a} and \ref{thm 1b}, take $\symm AB$ an almost complete balanced pair. Then, the upper and lower chambers of $g\!\left(\symm AB\right)$ give the two completions of $\symm AB$. The last vertex of the upper chamber will be the $g$-vector of some $\symm{A_0}{B_0}$ and the last vertex of the lower chamber will be the $g$-vector of some $\symm{A_0'}{B_0'}$. See Figure \ref{Fig: almost complete A/B}. See also Remark \ref{rem: almost complete A/B}.

{
We also have the following converse of Corollaries \ref{cor: vertices of U+} and \ref{cor: vertices of U0}.

\begin{Theorem}[Theorem \ref{thm: converse of upper/lower}]\label{thm C}
Let $\Lambda$ be a f.d. algebra. Let $\symm AB$ be a complete balanced pair where $A,B$ are both nonzero. Let $\cU=\cC\left(\symm AB\right)$ be the corresponding standard open chamber. Then there exists a unique decomposition
\[
	\symm AB=\symm{A'}{B'}\sqcup \symm{A''}{B''}
\]
where $\cU$ is the upper chamber of $g\!\left(\symm{A'}{B'}\right)$ and $\cU$ is the lower chamber of $g\!\left(\symm{A''}{B''}\right)$, i.e., $\symm AB$ is the generalized Bongartz completion of $\symm{A''}{B''}$ and the generalized co-Bongartz completion of $\symm{A'}{B'}$.
\end{Theorem}

}

For an indecomposable balanced pair $\symm{A_0}{B_0}$, all but one of the walls of the upper chamber $\cB_-$ of $v_0=g\!\left(\symm {A_0}{B_0}\right)$ are ``green'' and contain $v_0$. The wall of $\cB_-$ not containing $v_0$ is called the ``covering wall'' (Def. \ref{def: covering domain}) if it is red (and this happens if and only if $A_0\neq0$). The wall of the lower chamber $\cB_+$ not containing $v_0$ is called the ``flooring wall'' (Def. \ref{def: flooring domain}) if it is green (and this happens if and only if $B_0\neq0$). This is illustrated in Figure \ref{Fig02: A3 example}.

\begin{Theorem}[Theorem \ref{thm: A_0=M_0, hereditary}]\label{thm 2}
For the semi-invariant picture for a hereditary algebra $H$, the covering wall and the flooring wall for the $g$ vector $g\!\left(\symm AB\right)$ of an indecomposable balanced pair are $\cD_H(A)$ and $\cD_H(B)$ respectively.
\end{Theorem}

{
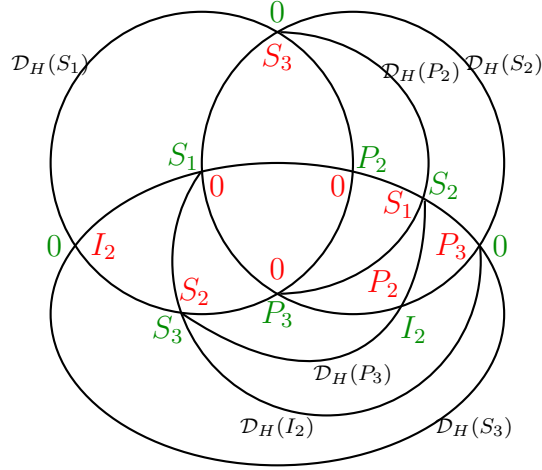
\begin{figure}[htbp]
\begin{center}
\begin{tikzpicture}

\begin{scope}
\coordinate (A) at (-1.25,-2);
\coordinate (B) at (1.95,-.5);
\draw[thick] (0,0) ellipse[x radius= 2cm,y radius=1.73cm];
\draw[thick][fill,white] (-1,0) circle[radius=2cm];
\draw[thick] (-1,0) circle[radius=2cm];
\draw[thick] (1,0) circle[radius=2cm];
\draw[thick] (0,-2) ellipse[x radius=3cm,y radius=2cm];
\draw[thick] (A)..controls (1,-3.5) and (2,-2)..(B);
\coordinate (D1) at (2.6,-3.5);
\draw (D1) node{\tiny$\cD_H(S_3)$};
\coordinate (D2) at (-3,1.3);
\coordinate (D2r) at (3,1.3);
\draw (D2) node{\tiny$\cD_H(S_1)$};
\draw (D2r) node{\tiny$\cD_H(S_2)$};

\coordinate (D3) at (0,-3.5);
\draw (D3) node{\tiny$\cD_H(I_2)$};

\coordinate (D4) at (1,-2.8);
\coordinate (D4u) at (1.9,1.2);
\draw (D4) node{\tiny$\cD_H(P_3)$};
\draw (D4u) node{\tiny$\cD_H(P_2)$};

\draw[black!45!green](0,-2) node{$P_3$};
\draw[red](0,-1.4) node{$0$};
\draw[black!45!green](0,2) node{$0$};
\draw[red](0,1.4) node{$S_3$};
\draw[black!45!green](1.25,0.05) node{$P_2$};
\draw[red](0.8,-.3) node{$0$};
\draw[black!45!green](-1.25,0.08) node{$S_1$};
\draw[red](-0.8,-.3) node{$0$};
\draw[black!45!green](2.95,-1.1) node{$0$};
\draw[red](2.3,-1.1) node{$P_3$};
\draw[black!45!green](-2.95,-1.1) node{$0$};
\draw[red](-2.3,-1.1) node{$I_2$};
\draw[black!45!green](-1.45,-2.2) node{$S_3$};
\draw[red](-1.1,-1.7) node{$S_2$};
\draw[black!45!green](2.2,-.3) node{$S_2$};
\draw[red](1.6,-.55) node{$S_1$};
\draw[black!45!green](1.8,-2.2) node{$I_2$};
\draw[red](1.4,-1.6) node{$P_2$};
\clip (0,-2) ellipse[x radius=3cm,y radius=2cm];
\draw[thick] (.65,-1.3) circle[radius=2.04cm];
\end{scope}
\end{tikzpicture}
\caption{This is Figure \ref{Fig01: A3 example} with red labels added on the opposite side from the green label as described in the caption of Figure \ref{Fig02: A3 example}. The green labels $P_i[1]$ on some of the vertices in Figure \ref{Fig01: A3 example} are replaced with the balanced pairs $\symm 0{I_i}$ as prescribed by Theorem \ref{thm 2}.
}
\label{Fig03: A3 example}
\end{center}
\end{figure}
}

For example, in Figure \ref{Fig03: A3 example}, the balanced notation for the shifted projective $P_j[1]$ is $\symm 0{I_j}$ and, by Theorem \ref{thm 2}, this corresponds to the fact that this vertex has no covering wall and flooring wall $\cD_H(I_j)$. Theorem \ref{thm 2} was one of the motivations for balanced notation. Figure \ref{Fig01: A3 example} shows the standard labeling of the semi-invariant picture for the hereditary algebra of type $A_3$. Figure \ref{Fig03: A3 example} shows the new labeling with balanced pairs.

\begin{Theorem}[Propositions \ref{prop: radical trace}, \ref{prop: C cap C0 is M+}]\label{thm 3}
In general, the covering wall $\cD_\Lambda(M_-)$ of $g\!\left(\symm {A_0}{B_0}\right)$ has $M_-$ a quotient of $A_0$ and the flooring wall is $\cD_\Lambda(M_+)$ where $M_+$ is a submodule of $B_0$.
\end{Theorem}

} 

}


{
\section{Basic Concepts}\label{sec1: basic concepts}

Let $\Lambda$ be a finite dimensional algebra over a field $K$ and let $mod\text-\Lambda$ be the category of finitely generated right $\Lambda$-modules. Let $S_1,\cdots,S_n$ be the simple modules with projective covers $P_1,\cdots,P_n$ and injective envelopes $I_1,\cdots,I_n$. 

The \textbf{dimension vector} of a $\Lambda$-modules $M$ is the integer vector $\undim M=(d_1,\cdots,d_n)$ where $d_i$ is the number of times that $S_i$ occurs in the composition series of $M$. Thus, the dimension of $M$ over the ground field $K$ is given by the dot product:
\[
    \dim_KM=(\dim_KS_1,\dim_KS_2,\cdots,\dim_KS_n)\cdot \undim M.
\]

The \textbf{$g$-vector} $g(M)$ of $M$ is defined by the equation
\[
	g(M):=\undim (P_0/rP_0)-\undim(P_1/rP_1)
\]
where $P_1\to P_0\to M$ is the minimal projective presentation of $M$. We note that the \textbf{top} of $M$, given by $\topp M:=M/rM$ is isomorphic to $\topp P_0=P_0/rP_0$. 

To account for the possibility of a valued quiver such as type $B_n,C_n,F_4,G_2$, where simple modules have different dimension, we use the \textbf{modified dot product}:
\begin{equation}\label{eq: modified dot product}
        v\ast w:=\sum_i v_iw_i\dim_KS_i \quad \text{ for } v,w\in \RR^n.
\end{equation}
\begin{Remark}\label{rem: ast notation}
\begin{enumerate}
    \item[(a)] $\dim_KM=\eta\ast \undim M$ where $\eta=(1,1,\cdots,1)$.
\item[(b)] The standard dot product counts each simple as $1$. So, $\eta\cdot\undim M$ is the \textbf{length} of $M$ (the number of simple modules in its composition series).
\item[(c)] For any projective module $P$ and any $M$ we have
\[
\dim_K\Hom_\Lambda(P,M)=\undim\topp P\ast\undim M=g(P)\ast\undim M.
\]
\end{enumerate}
\end{Remark}
If $M$ is not projective, the \textbf{Auslander-Reiten translate} $\tau M$ of $M$ is defined to be the kernel of the map of injective modules
\[
	\nu P_1\to \nu P_0
\]
induced by the projective presentation $P_1\to P_0\to M$ of $M$. Here $\nu$ is the \textbf{Nakayama functor} defined by
\[
	\nu M:=D\Hom_\Lambda(M,\Lambda)
\]
where $D$ denotes the $K$-dual: $DV:=\Hom_K(V,K)$ which transform left $\Lambda$-modules to right $\Lambda$-modules. The Nakayama functor has the property that it takes projective modules to injective modules and
\[
    \soc \nu P\cong\topp P.
\]

\begin{Proposition}\label{prop: g(M)} A balanced formula for the $g$-vector of $M$ is given by
\begin{equation}\label{eq: balance formula for g-vector}
	g(M)=\undim \topp M-\undim \soc \tau M.
\end{equation}
This holds even when $M$ is projective since, in that case, $\tau M=0$.
\end{Proposition}
\begin{proof}
   The standard $g$-vector of $M$ is $\undim\topp P_0-\undim\topp P_1$. But $\topp P_0\cong \topp M$ and $\soc\tau M$ is isomorphic to the socle of the injective envelope of $\tau M$ which is $\nu P_1$. This, in turn, is the top of $P_1$. 
\end{proof}

\section{Balanced notation for $\tau$-rigid pairs}\label{sec2:balanced notation}

In $\tau$-tilting theory we consider \textbf{$\tau$-rigid pairs} which are defined to be pairs of finitely generated modules $(M,P)$ so that
\begin{enumerate}
    \item $M$ is \textbf{$\tau$-rigid}, meaning $\Hom_\Lambda(M,\tau M)=0$ and
    \item $P$ is projective so that $\Hom_\Lambda(P,M)=0$.
\end{enumerate}
Note that condition (1) holds vacuously when $M$ is projective.
We assume both $M$ and $P$ do not have repeated summands. We often write the $\tau$-rigid pair $(M,P)$ as a single object in the derived category of $mod\text-\Lambda$, namely $M\sqcup P[1]$ where $M,P$ satisfy conditions (1) and (2) above. We note that
\[
    \tau\left(M\,{\textstyle\sqcup}\, P[1]\right)=\tau M\,{\textstyle\sqcup}\, \nu P.
\]
If the number of summands of $M\sqcup P[1]$ is $n$, the rank of $\Lambda$, then it is maximal since the $g$-vectors of the direct summands of $M\sqcup P[1]$ are linearly independent \cite[Theorem 5.1]{AIR}. (The $g$-vector of $P[1]$ is defined to be $g(P[1])=-g(P)$.)

\begin{Definition}\label{def: balanced notation A/B}
    Given $A,B\in mod\text-\Lambda$, we define $\symm AB$ to be a \textbf{balanced pair} if 
    \begin{enumerate}
        \item $\Hom_\Lambda(A,B)=0$
        \item $B\cong\tau A\sqcup  I$ where $I$ is injective and
        \item $A\cong\tau^{-1}B\sqcup P$ where $P$ is projective
    \end{enumerate}
    The isomorphisms in (2) and (3) exist but are not specified. 

    \end{Definition}
    
    \begin{Lemma}\label{lem: balanced of (M,P)} 
There is a bijection between the sets of $\tau$-rigid pairs $(M,P)$ and balanced pairs $\symm AB$ given as follows.
    \begin{enumerate}
    \item The $\tau$-rigid pair $(M,P)$ corresponds to the balanced pair $\symm M{\tau M\sqcup \nu P}$.
    \item The balanced pair $\symm AB$ corresponds to the $\tau$-rigid pair $(A,\nu^{-1}I)$ where $I$ is the sum of the injective components of $B$.
    \end{enumerate}
   \end{Lemma}

\begin{proof}
Given a $\tau$-rigid pair $(M,P)$, $\symm M{\tau M\sqcup \nu P}$ is a balanced pair since $B=\tau M\sqcup \nu P$ is the sum of $\tau M$ with an injective module and $\Hom_\Lambda(M,\tau M\sqcup \nu P)=\Hom_\Lambda(M,\tau M)\sqcup D\Hom_\Lambda(P,M)=0$.

Given a balanced pair $\symm AB$, $B=\tau A\sqcup I$ for a unique injective module $I$ with $\Hom_\Lambda(A,I)=0$. In that case $\Hom_\Lambda(\nu^{-1}I,A)=0$ and $(A,\nu^{-1}I)$ is a $\tau$-rigid pair.

To show that this is a bijection, start with $(M,P)$. (1) gives us $\symm AB$ where $A=M$ and $B=\tau M\sqcup \nu P$ with injective component $\nu P$. Apply (2) we get $(A,\nu^{-1}\nu P)=(M,P)$.

If we start with a balanced pair $\symm AB$, (2) gives us $(A,\nu^{-1}I)$ where $B=\tau A\sqcup I$. Then (1) gives us $\symm A{\tau A\sqcup \nu\nu^{-1}I}=\symm A{\tau A\sqcup I}\cong \symm AB$. So, we have a bijection.
\end{proof}

We define the \textbf{$g$-vector} of a balanced pair to be:
\[
g\!\left(\symm AB
\right):= \undim \topp A-\undim \soc B.
\]
\begin{Proposition}\label{prop: balanced g-vector}
    If $\symm AB$ is the balanced notation for $M\sqcup P[1]$, the formula for the $g$-vector agrees with the standard formula, i.e.,
    \[
    g\!\left(\symm AB
\right)=g(M)-g(P).
    \]
\end{Proposition}

\begin{proof}
    By Lemma \ref{lem: balanced of (M,P)}, $A=M$ and $B=\tau M\sqcup \nu P$. Given a minimal projective presentation $P_1\to P_0\to M$, we get a minimal injective copresentation $\tau M\to \nu P_1\to \nu P_0$ of $\tau M$. And, $\undim \topp P_1=\undim \soc \nu P_1=\undim \soc \tau M$. So,
    \[
        g(M)=\undim \topp P_0-\undim \topp P_1=\undim \topp M-\undim \soc \tau M,
    \]
    \[
        g(P)=\undim \topp P=\undim \soc \nu P.
    \]
    Thus
    \[
    g(M\sqcup P[1])=g(M)-g(P)=\undim \topp M-\undim \soc \tau M-\undim \soc\nu P
    \]
    
    \qquad\qquad\qquad\qquad$=\undim \topp A-\undim \soc B=g\!\left(\symm AB
\right).$
\end{proof}


\section{Jasso category and stability diagrams}\label{sec3: Jasso cat}

The \textbf{Jasso category} of a $\tau$-rigid pair $(M,P)$ is $J(M,P)=M^\perp\cap P^\perp\cap\,^\perp\tau M$. We give another formula for the same category using the balanced notation.

\begin{Proposition}\label{thm: Jasso category}
    The Jasso category in balanced notation is
    \[
    J\!\left(\symm AB\right)=A^\perp\cap\,^\perp B.
    \]
\end{Proposition}

\begin{proof}
    By Lemma \ref{lem: balanced of (M,P)}, $A=M$ and $B=\tau M\sqcup \nu P$. Then, since $P^\perp=\,^\perp\nu P$,\vs2

\qquad\qquad\qquad{$A^\perp\cap\,^\perp B=M^\perp\cap \,^\perp(\nu P\textstyle\sqcup \tau M)=M^\perp\cap P^\perp \cap \,^\perp \tau M.$}
\end{proof}

{\begin{Definition}\label{def: DL(M)}
    For any module $M$, we define the \textbf{semi-invariant domain} 
    \[
    \cD_\Lambda(M):=\{v\in\RR^n\mid v\ast\undim M=0 \text{ and }v\ast\undim M'\le0 \ \forall M'\subseteq M\}
    \]
    where the modified dot product ($\ast$) was defined in \eqref{eq: modified dot product}. This terminology comes from \cite{IOTW2} where, for $\Lambda$ hereditary, the semi-invariant $\det p^\ast$ is defined only when, for $p:P_1\to P_0$ a morphism of projective modules, the induced map
    \[
    	p^\ast:\Hom_\Lambda(P_0,M)\to \Hom_\Lambda(P_1,M)
    \]	
    is an isomorphism allowing us to take its determinant with respect to some basis and this is only possible if $\undim \topp P_0-\undim \topp P_1\in \cD_\Lambda(M)$. (Proposition \ref{prop: exact sequence for A/B} and Lemma \ref{lem: theta iso on W(A/B)} show that this makes sense even when $\Lambda$ is not hereditary.)
    
      We also define
    \[
    \cW(v):=\{M\in mod\text-\Lambda\mid v\in \cD_\Lambda(M)\}.
    \]
\end{Definition}
}

\begin{Definition}\cite{IngTh}\label{def: wide subcategorhy}
    A \textbf{wide subcategory} of $mod\text-\Lambda$ is an abelian category closed under extensions in $mod\text-\Lambda$ and \textbf{exactly embedded} in $mod\text-\Lambda$. This is equivalent to the statement that it contains the kernel, image and cokernel of any morphism between two of its objects.
\end{Definition}

\begin{Remark}\label{rem: W(v) is wide} \begin{enumerate}
\item It is a well-known fact that $\cW(v)$ is a {wide subcategory}. Call $\cW(v)$ the \textbf{wide subcategory associated to $v$}.
\item $M\in \cW(v)\Leftrightarrow v\in \cD_\Lambda(M).$
\end{enumerate}
\end{Remark}

\begin{Remark}\label{rem: rational points are dense}
    It is clear from Def. \ref{def: DL(M)} that $\cD_\Lambda(M)$ is a closed convex subset of $\RR^n$. Also, $\cD_\Lambda(M)$ is given by a finite set of linear equations and inequalities with rational coefficients. By linear algebra over the rational numbers we see that
    \begin{enumerate}
        \item The set of rational points in $\cD_\Lambda(M)$, i.e., $\cD_\Lambda(M)\cap \QQ^n$, is dense in $\cD_\Lambda(M)$. 
        \item This also holds for any finite intersection, i.e., the rational points in $\bigcap \cD_\Lambda(M_i)$ form a dense subset of $\bigcap \cD_\Lambda(M_i)$.
        \item In particular, each corner of $\cD_\Lambda(M)$ is the ray generated by one rational vector.
        \item So, $\cD_\Lambda(M)$ and each of its faces is spanned by a dense set of rational points.
    \end{enumerate}
\end{Remark}


To prove Theorem \ref{thm: W(g)=J}, we need the following well-known fact.
 \begin{Lemma}\label{lem: duality of nu P}
     For every module $X$ and projective module $P$ we have a natural isomorphism $\theta:\Hom_\Lambda(P,X)\cong D\Hom_\Lambda(X,\nu P)$.
 \end{Lemma}

 \begin{proof}      It is clear that $\Hom_\Lambda(P,X)\cong D\Hom_\Lambda(X,\nu P)$ since both sides are additive in $P$ and, for indecomposable $P$, they both count the number of times $\topp P=\soc \nu P$ appears in the composition series of $X$. We need a formula for this isomorphism so that we can check that it is natural.

{
There is a natural isomorphism $X\otimes_\Lambda \Hom_\Lambda(P,\Lambda)\xrightarrow{\theta_1} \Hom_\Lambda(P,X)$ obtained as a composition
    \[
    X\otimes_\Lambda \Hom_\Lambda(P,\Lambda)\cong \Hom_\Lambda(\Lambda,X)\otimes_\Lambda \Hom_\Lambda(P,\Lambda)\xrightarrow{\circ} \Hom_\Lambda(P,X)
    \]
    which is surjective since $P$ is projective and is an isomorphism since domain and range have the same size. Also, there is a natural isomorphism
    \[
        X\otimes_\Lambda \Hom_\Lambda(P,\Lambda)\xrightarrow{\theta_2}D\Hom_\Lambda(X,\nu P)
    \]
    which is dual to the adjunction isomorphism:    \[
    D(X\otimes_\Lambda \Hom_\Lambda(P,\Lambda))
    \cong \Hom_\Lambda(X, D\Hom_\Lambda(P,\Lambda))=\Hom_\Lambda(X,\nu P).
    \]
So, we have
    \xymatrix{
\quad\theta: \Hom_\Lambda(P,X)\ar[r]^{\theta_1^{-1}}_\approx& 
	X\otimes_\Lambda\Hom_\Lambda(P,\Lambda) \ar[r]^{\theta_2}_\approx&
	D\Hom_\Lambda(X,\nu P).
	}
}
 \end{proof}
 
 We also need an exact sequence which is a balanced version of \cite[Proposition 2.4]{AIR}.
\begin{Proposition}\label{prop: exact sequence for A/B}
    Given a balanced pair $\symm AB$, let $p:Q\to A$ be the projective cover of $A$ and let $j:B\to E$ be the injective envelope of $B$. Then, for any module $X$ we have an exact sequence of the following form where $\delta$ is described in the proof:
    \[
          0\to \Hom_\Lambda(A,X)\xrightarrow{p^\ast} \Hom_\Lambda(Q,X)\xrightarrow{\delta} D\Hom_\Lambda(X,E)\xrightarrow{j_\ast} D\Hom_\Lambda(X,B)\to 0.
    \]
\end{Proposition}

\begin{proof} This is the proof from \cite{AIR} in balanced notation. It uses only the fact that $A=A_0\sqcup P$ and $B=B_0\sqcup I$ where $P$ is projective, $I$ is injective and $B_0=\tau A_0$. 

Let $P_1\xrightarrow{p}P_0\to A_0$ be the minimal projective presentation of $A_0$. Then $B_0\to \nu P_1\xrightarrow{\nu p}\nu P_0$ is the minimal injective copresentation of $B_0$. For any $\Lambda$-module $X$, we get the following commuting diagram with exact rows where the vertical maps form the natural isomorphism given in Lemma \ref{lem: duality of nu P}.
    \[
\xymatrixrowsep{15pt}\xymatrixcolsep{10pt}
\xymatrix{
0\ar[r]& \Hom_\Lambda(A_0,X)\ar[r] &
	\Hom_\Lambda(P_0,X)\ar[d]_\theta^\approx\ar[r]^{p^\ast}\ar[dr]^\delta &
	\Hom_\Lambda(P_1,X)\ar[d]_\theta^\approx
	\\
 &   & 
	D\Hom_\Lambda(X,\nu P_0) \ar[r]^{p'}&
	D\Hom_\Lambda(X,\nu P_1)\ar[r]& D\Hom_\Lambda(X,B_0)\to 0
	} 
    \]
     where $p'=D\Hom_\Lambda(X,\nu p)$. This gives the natural exact sequence
  \[
    0\to \Hom_\Lambda(A_0,X)\to \Hom_\Lambda(P_0,X)\xrightarrow{\delta}D\Hom_\Lambda(X,\nu P_1)\to D\Hom_\Lambda(X,B_0)\to 0.
    \]
    Since the projective cover of $P$ is $P$ and the injective envelope of $I$ is $I$, we also have the following exact sequence
    \[
    0\to \Hom_\Lambda(P,X)\xrightarrow= \Hom_\Lambda(P,X)\xrightarrow{\delta=0} D\Hom_\Lambda(X,I)\xrightarrow= D\Hom_\Lambda(X,I)\to 0.
    \]
    Since $Q=P_0\sqcup P$ is the projective cover of $A=A_0\sqcup P$ and $E=\nu P_1\sqcup I$ is the injective envelope of $B=B_0\sqcup I$, the sum of the last two 4-term exact sequences gives the desired exact sequence.
\end{proof}

{\begin{Corollary}\label{cor: (X,B)=0 then E(A,X)=0} \label{cor: (A,Z)=0 then E(Z,B)=0}
Let $\symm AB$ be a balanced pair and $X$ a $\Lambda$-module. 

\emph{(a)} If $\Hom_\Lambda(X,B)=0$ then $\Ext_\Lambda^1(A,X)=0$.

\emph{(b)} If $\Hom_\Lambda(A,X)=0$ then $\Ext_\Lambda^1(X,B)=0$.
\end{Corollary}

\begin{proof}
   (a) If $\Hom_\Lambda(X,B)=0$, then, using the commutative diagram in the proof above, we obtain a short exact sequence
    \[
        0\to \Hom_\Lambda(A_0,X)\to \Hom_\Lambda(P_0,X)\to \Hom_\Lambda(P_1,X)\to 0.
    \]
    So, $\Ext_\Lambda^1(A_0,X)=0$ and, therefore, $\Ext_\Lambda^1(A,X)=\Ext_\Lambda^1(A_0\sqcup P,X)=0$. 
    
    (b) is similar.
\end{proof}
}

\begin{Corollary}\label{cor: g(A/B) dot X}
    Let $\symm AB$ be a balanced pair. For any $\Lambda$-module $X$ we have
    \[
    g\!\left(\symm AB\right)\ast\undim X=\dim_K \Hom_\Lambda(A,X)-\dim_K \Hom_\Lambda(X,B).
    \]
\end{Corollary}

\begin{proof} Let $Q$ be the projective cover of $A$ and $E$ the injective envelope of $B$. Then, by Remark \ref{rem: ast notation}(c), it follows that
   \begin{eqnarray*}
g\!\left(\symm AB\right)\ast\undim X&=\undim\topp A\ast \undim X-\undim \soc B\ast \undim X\\
	&= \undim \topp Q\ast \undim X-\undim \soc E\ast \undim X\\
	&=\dim_K \Hom_\Lambda(Q,X)-\dim_K\Hom_\Lambda(X,E)\\
	&=\dim_K \Hom_\Lambda(A,X)-\dim_K \Hom_\Lambda(X,B)
\end{eqnarray*}
where the last equation follows from exact sequence in Proposition \ref{prop: exact sequence for A/B}. 
\end{proof}

\begin{Lemma}\label{lem: theta iso on W(A/B)} 
Let $\symm AB$ be a balanced pair, $X$ a $\Lambda$-module. The following are equivalent.
\begin{enumerate}
    \item The $\Lambda$-module $X$ is in the Jasso category $A^\perp\cap\,^\perp B$.
    \item The morphism $\delta: \Hom_\Lambda(Q,X)\to \Hom_\Lambda(X,E)$ from Proposition \ref{prop: exact sequence for A/B} is an isomorphism.
    \item 
    \begin{enumerate}
            \item $\dim_K \Hom_\Lambda(Q,X)=\dim_K \Hom_\Lambda(X,E)$ and
    \item $\dim_K \Hom_\Lambda(Q,X')\le \dim_K \Hom_\Lambda(X',E)$ for any submodule $X'$ of $X$.
    \end{enumerate}
\end{enumerate}
\end{Lemma}

\begin{proof}
 By Proposition \ref{prop: exact sequence for A/B} we have the exact sequence:
\[
        0\to \Hom_\Lambda(A,X)\xrightarrow{p^\ast} \Hom_\Lambda(Q,X)\xrightarrow{\delta} D\Hom_\Lambda(X,E)\xrightarrow{j_\ast} D\Hom_\Lambda(X,B)\to 0.
    \]
    
    $(1)\ifff (2)$: $X\in A^\perp\cap\,^\perp B$ if and only if $\Hom_\Lambda(A,X)=0=\Hom_\Lambda(X,B)$ if and only if $\delta$ is an isomorphism.

$(2)\then(3)$  (3a) holds since $\delta$ is an isomorphism. To show (3b), note that $X\in A^\perp$ implies that $X'\in A^\perp$ for all submodules $X'\subset X$. So, we have:    \[
        0\to \Hom_\Lambda(Q,X')\to D\Hom_\Lambda(X',E)\to D\Hom_\Lambda (X',B)\to 0.
    \]
    This exact sequence implies $(3b)$.
    
   $(3)\then(1)$ Suppose (3a) and (3b) hold but $X\notin A^\perp\cap\,^\perp B$. Then $X\notin A^\perp$ since, otherwise, $\delta$ would by a monomorphism and, by (3a), this would imply that $\delta$ is an isomorphism. Since $X\notin A^\perp$, there is a nonzero morphism $f:A\to X$. Let $X'$ be the image of $f$. Then $A\onto X'$ implies $\Hom_\Lambda(X',B)=0$ since $\Hom_\Lambda(A,B)=0$. The exact sequence in Proposition \ref{prop: exact sequence for A/B} for $X'$ becomes the following short exact sequence:
    \[
    0\to \Hom_\Lambda(A,X')\to \Hom_\Lambda(Q,X')\to  D\Hom_\Lambda(X',E)\to 0.
    \]
    So, $\dim_K \Hom_\Lambda(Q,X')>\dim_K \Hom_\Lambda(X',E)$ contradicting (3b). Thus, (3a) and (3b) imply (1) $X\in A^\perp\cap\,^\perp B$. So, (1), (2), (3) are equivalent.
 \end{proof}

We now state a well known theorem about the relation between Jasso categories and wide subcategories and give a short balanced proof.

\begin{Theorem}\label{thm: W(g)=J}
    Let $\symm AB$ be a balanced pair. Then the Jasso category $J\!\left(\symm AB\right)$ is equal to the wide subcategory associated to the $g$-vector $g\!\left(\symm AB\right)$, i.e., $$J\!\left(\symm AB\right)=\cW\!\left(g\!\left(\symm AB\right)\right).$$
\end{Theorem}

\begin{proof}
The statement we want to prove is that $\cW(g(\tsymm AB))=A^\perp\cap\,^\perp B$. So, let $X\in \cW(g(\tsymm AB))$ which, by definition of $\cW(g(\tsymm AB))$ is equivalent to $g(\tsymm AB)\in\cD_\Lambda(X)$. Then, by definition of $\cD_\Lambda(X)$, we have:
\vs2

 (1) $g(\tsymm AB)\ast \undim X=0$ and
 (2) $g(\tsymm AB)\ast \undim X'\le 0$ for all submodules $X'\subseteq X$.
 \vs2
 
\noindent By Corollary \ref{cor: g(A/B) dot X}, (1) is equivalent to Lemma \ref{lem: theta iso on W(A/B)}(3a) and (2) is equivalent to Lemma \ref{lem: theta iso on W(A/B)}(3b). Thus by Lemma \ref{lem: theta iso on W(A/B)}(1), we have $X\in A^\perp\cap\,^\perp B$.
 
 To prove the converse, suppose $X\in A^\perp\cap\,^\perp B$. Then, we get conditions (3a) and (3b) in Lemma \ref{lem: theta iso on W(A/B)}. These are equivalent to (1) and (2) above. So, $X\in\cW(g(\tsymm AB))$. 
 \end{proof}

 \begin{Corollary}\label{cor: M in J, g in D(M)}
     Let $\symm AB$ be a balanced pair and $M$ any module. Then 
     \[M\in J\!\left(\symm AB\right)\iff g\!\left(\symm AB\right)\in \cD_\Lambda(M).\]
 \end{Corollary}


\section{Decomposition of balanced pairs}\label{sums of pairs}

\begin{Definition}\label{def: direct sum of balanced pairs}
    By a \textbf{direct sum decomposition} of a balanced pair $\symm AB$ we mean a pair of direct sum decompositions $A=\bigsqcup A_i$, $B=\bigsqcup B_i$ with summands allowed to be zero, so that each $\symm {A_i}{B_i}$ is a nonzero balanced pair.
\end{Definition}

{
\begin{Proposition}\label{prop: J(sum)}
\emph{(a)} There are three kinds of indecomposable balanced pairs:
\[
    \symm A{\tau A},\quad \symm P0,\quad \symm 0I
\]
where $A$ is an indecomposable nonprojective $\tau$-rigid module, $P$ is an indecomposable projective module and $I$ is an indecomposable injective module.

\emph{(b)} 
\[
    J\!\left(\bigsqcup\symm {A_i}{B_i}\right)=\bigcap  J\!\left(\symm {A_i}{B_i}\right),\qquad
    J\!\left(\bigsqcup\symm {A_i}{B_i}\right)=J\!\left(\bigsqcup\symm {A_i^{c_i}}{B_i^{c_i}}\right)
    \]
    for any sequence of positive integers $\{c_i\}$.

\emph{(c)} 
\[
    g\!\left(\bigsqcup\symm {A_i^{c_i}}{B_i^{c_i}}\right)=\sum c_ig\!\left(\symm{A_i}{B_i}
    \right).
\]
\end{Proposition}

\begin{proof}
(b) follows from $\bigsqcup\symm {A_i}{B_i}=\symm{\sqcup A_i}{\sqcup B_i}$ and 
\[
(\sqcup A_i)^\perp\cap \,^\perp(\sqcup B_i)=\bigcap A_i^\perp\cap\bigcap\,^\perp B_i=\bigcap\, (A_i^\perp\cap\,^\perp B_i).
\]
(c) follow from the fact that $g(X\sqcup Y)=g(X)+g(Y)$.
\end{proof}
}

\begin{Theorem}\label{thm: the open simplex}
    Let $\symm AB=\bigsqcup_{i=0}^m \symm{A_i}{B_i}$ be a direct sum decomposition of a balanced pair and let $\{t_i>0\}$ be any sequence of positive real numbers. Then:
    \begin{equation}\label{eq: convexity of walls}
      \cW\!\left(\sum_{i=0}^m t_ig\!\left(\symm {A_i}{B_i}\right)\right)=J\!\left(\symm AB\right).
    \end{equation}
\end{Theorem}

\begin{proof}
Let $M\in \cW(\sum t_ig(\tsymm{A_i}{B_i}))$. This is equivalent to 
\[
\sum t_ig\!\left(\symm{A_i}{B_i}\right)\in\cD_\Lambda(M).\]
It follows from the definition of $\cD_\Lambda(M)$ and $\cW(v)$ that $\cW(tv)=\cW(v)$ for any $v\in \RR^n$ and any $t>0$, $t\in\RR$. 

Case 1. All $t_i\in\QQ$. Then, we may scale them up to integers, say $c_i$. Since $t_i>0$ we get $c_i>0$. Then 
\[
	\cW\!\left(\sum t_ig\!\left(\symm{A_i}{B_i}\right)\right)=\cW\!\left(\sum c_ig\!\left(\symm{A_i}{B_i}\right)\right)=\cW\!\left(g\!\left(\bigsqcup\symm{A_i^{c_i}}{B_i^{c_i}}\right)\right)=J\!\left(\bigsqcup \symm{A_i^{c_i}}{B_i^{c_i}}\right)=J\!\left(\symm AB\right)
\]
by Theorem \ref{thm: W(g)=J} and Proposition \ref{prop: J(sum)}. So, $M\in J(\tsymm AB)$ in the case when $t_i$ are all rational and positive.

Case 2. $t_i\in\RR$.
    Let $\Delta\subset \RR^n$ be the $m$-simplex spanned by the vectors $g(\tsymm{A_i}{B_i})$:
    \[
    \Delta:=\left\{\sum_{i=0}^m t_ig\!\left(\symm{A_i}{B_i}\right)\mid t_i\in\RR, t_i\ge0, \sum t_i=1\right\}.
    \]
    Let $L=\Delta\cap\cD_\Lambda(M)$.  Then $L$ is nonempty since $\sum t_ig(\tsymm{A_i}{B_i})\in L$. By Remark \ref{rem: rational points are dense}, $L$ contains a dense set of rational points. In particular, there are rational points $r$ arbitrarily close to $\sum t_ig(\tsymm{A_i}{B_i})$. We can choose $r=\sum s_ig(\tsymm{A_i}{B_i})$ where $s_i$ are positive rational numbers. Then, by Case 1, $\sum s_ig(\tsymm{A_i}{B_i})\in \cD_\Lambda(M)$ implies $M\in J(\tsymm AB)$. Therefore, 
    \[
    	\cW\!\left(\sum_{i=0}^m t_ig\!\left(\symm{A_i}{B_i}\right)\right)\subseteq J\!\left(\symm AB\right).
    \]
   
   Conversely, let $M\in J(\tsymm AB)$. Then, since $J(\tsymm AB)=\bigcap J(\tsymm{A_i}{B_i})$, 
\[
M\in J\!\left(\symm{A_i}{B_i}\right)=\cW\!\left(g\!\left(\symm{A_i}{B_i}\right)\right)
\]
for each $i$ by Theorem \ref{thm: W(g)=J}. By Remark \ref{rem: W(v) is wide} this is equivalent to $g(\tsymm{A_i}{B_i})\in \cD_\Lambda(M)$. Since $\cD_\Lambda(M)$ is convex, it contains the simplex spanned by the vectors $g(\tsymm{A_i}{B_i})$. Since $\cD_\Lambda(M)$ contains all positive scalar multiples of its elements, $\cD_\Lambda(M)$ contains any positive linear combination of the vectors $g(\tsymm{A_i}{B_i})$. So, $\sum t_ig(\tsymm{A_i}{B_i})\in \cD_\Lambda(M)$ which is equivalent to $M\in \cW(\sum t_ig(\tsymm{A_i}{B_i}))$. This proves:
\[
	J\!\left(\symm AB\right)= \cW\!\left(\sum_{i=0}^m t_ig\!\left(\symm{A_i}{B_i}\right)\right).
\]
    \end{proof}

{

\begin{Corollary}\label{cor: clean interiors}
Let $\symm AB=\bigsqcup_{i=0}^m\symm{A_i}{B_i}$ be a balanced pair with $m+1$ indecomposable summands. Let $\Delta$ be the closed $m$-simplex in $\RR^n$ with vertices $g(\tsymm{A_i}{B_i})$. Then, the following are equivalent for any module $M$.
\begin{enumerate}
    \item $\cD_\Lambda(M)\cap (interior\,\Delta)\neq\emptyset$.
    \item $\cD_\Lambda(M)$ contains $\Delta$.
    \item $M\in J\!\left(\symm AB\right)$.
    \item $M\in J\!\left(\symm{A_i}{B_i}\right)$ for all $i$.
\end{enumerate}\end{Corollary}

\begin{proof}
    (1)$\then$(3) Suppose $\cD_\Lambda(M)$ intersects the interior of $\Delta$. Then, $\sum t_ig(\tsymm{A_i}{B_i})\in\cD_\Lambda(M)$ for some real numbers $t_i>0$. By Remark \ref{rem: W(v) is wide} and Theorem \ref{thm: the open simplex}, this is equivalent to
    \[
         M\in \cW\!\left(\sum t_ig\!\left(\symm{A_i}{B_i}\right)\right)=J\!\left(\symm AB\right).
    \]
    
    (3)$\then$(4) is clear since $J(\tsymm AB)=\bigcap J(\tsymm{A_i}{B_i})$.
    
    (4)$\then$(2) $M\in J(\tsymm {A_i}{B_i})$ for every $i$ is equivalent to $g(\tsymm{A_i}{B_i})\in \cD_\Lambda(M)$ for all $i$ by Corollary \ref{cor: M in J, g in D(M)}. Since $\cD_\Lambda(M)$ is convex it contain the simplex $\Delta$ spanned by the points $g(\tsymm{A_i}{B_i})$. 
    
    Finally, (2) clearly implies (1). So, these statements are equivalent.
\end{proof}

We call a balanced pair $\symm AB$ \textbf{complete} if it has $n$ indecomposable nonisomorphic summands $\symm{A_i}{B_i}$ where $n$ is the rank of $\Lambda$ which is the number of simple $\Lambda$-modules. Theorem \ref{thm: the open simplex} can be rephrased as follows.

\begin{Corollary}\label{cor: empty interior}
Let $\symm AB=\bigsqcup_{i=0}^{n-1}\symm{A_i}{B_i}$ be a complete balanced pair. 

\emph{(a)} Let $\Delta$ be the closed $(n-1)$-simplex in $\RR^n$ with vertices $g(\tsymm{A_i}{B_i})$.
Then, there is no module $M$ so that $\cD_\Lambda(M)\cap(interior\,\Delta)\neq\emptyset$. 

\emph{(b)} $J\!\left(\symm AB\right)=0$.
\end{Corollary}

\begin{proof}
(a) If a domain $\cD_\Lambda(M)$ meets the interior of $\Delta$, it contains $\Delta$ by Corollary \ref{cor: clean interiors}, including the $n$ vertices of $\Delta$. But $\cD_\Lambda(M)$ also contains 0. So, it must be $n$-dimensional which is not possible since $\cD_\Lambda(M)$ is contained in the hyperplane perpendicular to $\undim M$.

(b) follows from (a) and Corollary \ref{cor: clean interiors}.
\end{proof}
}

{
We restate the following result due to Gustavo Jasso \cite{Jasso} (and \cite{DIRRT}) in balanced notation  since we will be using a corollary of it in the proof of Lemma \ref{lem: T is not A0} below.

\begin{Theorem}\cite{Jasso}\cite{DIRRT}\label{thm: J perpendicular}
Let $\symm AB$ be a balanced pair where $\symm AB$ has $m$ components. Then the Jasso subcategory $J\!\left(\symm AB\right)$ is equivalent to the module category of a finite dimensional algebra of rank $n-m$. In particular, it has $n-m$ (relatively) simple objects.
\end{Theorem}

\begin{Corollary}\label{cor: labels of domains}
When $\symm AB$ has $n-1$ components, $J\!\left(\symm AB\right)$ consists of a single object $M_0$ together with iterated self-extensions of $M_0$. In particular, every object of $J\!\left(\symm AB\right)$ has dimension vector a multiple of $\undim M_0$.
\end{Corollary}
}
{
\section{Torsion classes}\label{sec: P(v)}

We recall that, for any topological space $X$, we can define an equivalence relation by saying two points are equivalent if they are connected by a path in $X$. The equivalence classes are called the \textbf{path components} of $X$. These are the maximal path connected subsets of $X$. Let $\cD_\Lambda$ denote the union of all domains of semi-invariants $\cD_\Lambda(M)$. (We will call these \textbf{domains} for short.) A \textbf{chamber} is defined to be a path component of the space $X=\RR^n-\cD_\Lambda$. Since the origin $0$ lies on every domain $\cD_\Lambda(M)$, $0$ is not an element of any chamber. But domains and chambers share the following property.

\begin{Proposition}\label{prop: unions of rays}
    Domains and chambers are closed under positive scalar multiplication, i.e., if $x\in\cD_\Lambda(M)$, then $rx\in\cD_\Lambda(M)$ and if $v\in\cU$, then $rv\in \cU$ for all real numbers $r>0$.
\end{Proposition}

\begin{proof}
    This is clear for domains by definition of $\cD_\Lambda(M)$. For chambers, suppose by contradiction that $v\in \cU$ but $rv\in\cD_\Lambda(M)$ for some $r>0$. Then $v=\frac1r rv\in \cD_\Lambda(M)$, so $v\notin\cU$, a contradiction. Therefore, the entire open ray $\{rv\mid r>0\}$ lies in the complement of $\cD_\Lambda$. Since the open ray is a path connected set, all $rv$ lie in the same chamber $\cU$.
\end{proof}

If follows from Def. \ref{def: DL(M)} that domains $\cD_\Lambda(M)$ are closed convex sets. Chambers are never closed since the limit point $0$ is not in any chamber. However, they are convex. (This is part of Theorem \ref{thm: chambers are convex}.)

The chambers that have been studied the most are the ``{standard open chambers}'' defined as follows and illustrated in Figure \ref{fig: stardard chamber}.
{
\begin{Definition}\label{def: standard open chamber}
    Let $\symm AB=\bigsqcup\symm{A_i}{B_i}$ be a complete balanced pair. Then the \textbf{standard open chamber} $\cC\left(\symm AB\right)$ is the set of all positive linear combinations $\sum r_ig\!\left(\symm{A_i}{B_i}\right)$, $r_i>0$ of the $g$-vectors of $\tsymm{A_i}{B_i}$. 
\end{Definition}
}

{
\begin{Remark}\label{rem: projected domain, chamber}
   Let $\cC\left(\symm AB\right)$ be the standard open chamber associated to the complete balanced pair $\symm AB=\bigsqcup \symm{A_i}{B_i}$.
   
    (a) Then $\cC\left(\symm AB\right)$ contains the interior of the simplex $\Delta$ spanned by $g(\tsymm{A_i}{B_i})$ and is also the union of all open rays $\RR_+v$ of all vectors $v$ in the interior of $\Delta$. See Figure \ref{fig: stardard chamber}. 

(b) Because of Proposition \ref{prop: unions of rays} we can project domains and chambers to the unit sphere. The projection map is given by dividing vectors by their lengths. The images in $S^{n-1}$ of domains and chambers will be called \textbf{projected domains} and \textbf{projected chambers}. See Figure \ref{fig: projected chamber}. 
\end{Remark}
}

We recall that, for any $\Lambda$-module $M$, $\Gen M$ is the class of all quotients of multiples $M^m$ of $M$. We now need the following lemma of Auslander and Smal\o.

\begin{Lemma}\cite{AS}\label{lem: AS}   $\Hom_\Lambda(Y,\tau X)=0$ if and only if $\Ext_\Lambda^1(X,\Gen\,Y)=0$.\end{Lemma}

\begin{Proposition}\label{prop: Gen M is torsion class}
    If $M$ is $\tau$-rigid, then $\Gen M$ is a torsion class.
\end{Proposition}

\begin{proof}
    Clearly, $\Gen M$ is closed under quotients. To see that it is closed under extensions, consider a short exact sequence $0\to A\to B\to C\to 0$. If $A,B\in\Gen M$, we have epimorphisms $f:M^m\to A$ and $g:M^k\to C$. By Lemma \ref{lem: AS}, $\Ext_\Lambda^1(M,A)=0$. So, $g$ lists to $\widetilde g:M^k\to B$ giving an epimorphism $f+\widetilde g:M^{m+k}\to B$.
\end{proof}

{
\begin{Theorem}\cite{AIR}\cite{BST}\label{thm: bijection T - AB - C(AB)}
There are bijections between the following three sets.
\begin{enumerate}
\item The set of functorially finite torsion classes $\cT$ in $mod\text-\Lambda$
\item The set of complete balanced pairs $\symm AB$
\item The set of standard open chambers $\cC\left(\symm AB\right)$.
\end{enumerate}
The bijection $(3)\ifff(1)$ is given by $\cC\left(\symm AB\right)\leftrightarrow \Gen A$.
\end{Theorem}

\begin{proof}
The bijection $(1)\ifff(2)$ is \cite[Theorem 2.7]{AIR} via support $\tau$-tilting modules. The bijection $(1)\ifff(3)$ is by \cite[Proposition 3.27]{BST}.
\end{proof}
}

The projected domains and projected chambers are the same as the intersection of domains and chambers with the unit sphere.

{
\begin{figure}[htbp]
\begin{center}
\begin{tikzpicture}
\coordinate(A) at (0,0);
\coordinate(E1) at (2,1/3);
\coordinate(E2) at (1.8,1.8);
\coordinate(E3) at (2/3,2);
\draw (A)node{$\bullet$};
\draw (A) node[left]{$0$};
\draw[dashed] (0,0)--(3.6,.6);
\draw[dashed] (0,0)--(3.6,3.6);
\draw[dashed] (0,0)--(1.2,3.6);
\draw[thick,->] (0,0)--(3,.5);
\draw[thick,->] (0,0)--(3,3);
\draw[thick,->] (0,0)--(1,3);
\draw[very thick] (E1)--(E2)--(E3)--(E1);
\draw(E1)node[below]{$v_1$};
\draw(E2)node[right]{$v_2$};
\draw(E3)node[left]{$v_3$};
\draw (2,1) node[right]{$\Delta$};

\end{tikzpicture}
\caption{The vertices $v_i$ are $g\!\left(\symm{A_i}{B_i}\right)$ for a complete balanced pair $\symm AB=\bigsqcup\symm{A_i}{B_i}$. These $n$ vertices span the $(n-1)$-simplex $\Delta$ discussed in Corollary \ref{cor: empty interior}. The standard open chamber is the infinite open cone which is the union of all open rays through points in the interior of $\Delta$.}
\label{fig: stardard chamber}
\end{center}
\end{figure}
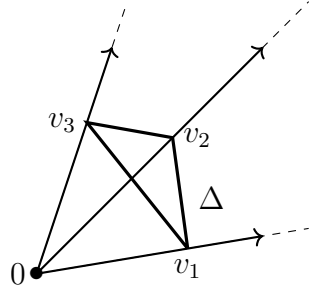
}

{
\begin{figure}[htbp]
\begin{center}
\begin{tikzpicture}
\coordinate(A) at (0,0);
\coordinate(E1) at (2,1/3);
\coordinate(E2) at (1.8,1.8);
\coordinate(E3) at (2/3,2);
\draw (A)node{$\bullet$};
\draw (A) node[left]{$0$};
\begin{scope}
\draw[thick,->] (0,0)--(3,3);
\draw[fill,white] (0,0) circle[radius=1.3cm];
\end{scope}
\draw[thick,->] (0,0)--(3,.5);
\draw[thick,->] (0,0)--(1,3);
\draw[thick] (0,0) circle[radius=1.5cm];
\draw[fill,white] (0,0) circle[radius=1.15cm];

\coordinate(X1) at (1.15,.2);
\coordinate(X2) at (.92,.9);
\coordinate(X22) at (.96,.85);
\coordinate(X3) at (.36,1.1);

\draw(E1) node{$\bullet$};
\draw(E2) node{$\bullet$};
\draw(E3) node{$\bullet$};
%
\begin{scope}[rotate=-49]
\clip (0,1.1) circle[radius=1.1cm];
\draw[very thick] (0,0) ellipse[x radius=15mm,y radius=11mm];
\end{scope}
\begin{scope}[rotate=4]
\clip (1,1.4) circle[radius=1cm];
\draw[very thick] (0,0) ellipse[x radius=15mm,y radius=11.2mm];
\end{scope}
\begin{scope}[rotate=-88]
\clip (-.6,1.4) circle[radius=1cm];
\draw[very thick] (0,0) ellipse[x radius=15mm,y radius=11.7mm];
\end{scope}
\draw[dashed] (0,0)--(3.6,.6);
\draw[dashed] (0,0)--(3.6,3.6);
\draw[dashed] (0,0)--(1.2,3.6);

\draw(E1)node[below]{$v_1$};
\draw(E2)node[right]{$v_2$};
\draw(E3)node[left]{$v_3$};
\end{tikzpicture}
\caption{We project domains $\cD_\Lambda(M)$ and standard open chambers $\cC(\tsymm AB)$ to the unit sphere $S^{n-1}$ by sending $v$ to $\frac{v}{||v||}$. The images of domains and chambers in $S^{n-1}$ are the ``projected domains'' and ``projected chambers''. Here the projected chamber is the interior of the displayed spherical simplex and the sides of the simplex are subsets of projected domains.}
\label{fig: projected chamber}
\end{center}
\end{figure}
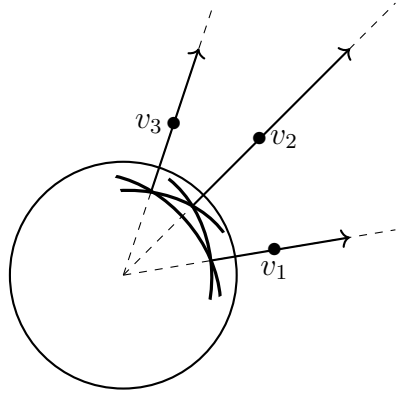
}
The domain $\cD_\Lambda(M)$ is called a \textbf{wall} if it is $(n-1)$-dimensional, i.e, of full dimension inside the hyperplane perpendicular to $\undim M$. 

\begin{Remark}\label{rem: domains and walls}
    For a hereditary algebra of finite type, the domains $\cD_\Lambda(M)$ of indecomposable modules $M$ are walls. In the tame case, we believe that $\cD_\Lambda(M)$ is a wall when $M$ is a \textbf{brick} (every nonzero endomorphism of $M$ is an isomorphism). However, in the wild case, there is no good characterization of modules $M$ so that $\cD_\Lambda(M)$ is a wall. 
\end{Remark}

{\begin{Example}\label{eg: pseudo-brick}
    Let $\Lambda=KQ$ be the path algebra of the quiver
    \[
    \xymatrixrowsep{10pt}\xymatrixcolsep{20pt}
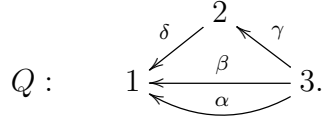
\xymatrix{
&&2\ar[dl]_\delta\\
Q:&1 && 3\ar[ul]_\gamma\ar[ll]_\beta\ar@/^1pc/[ll]_\alpha .
	} 
    \]
    Let $M$ be the module with dimension vector $(2,1,2)$ with
    \[
    M_\alpha=\mat{1 & 0\\0&1},\qquad M_\beta=\mat{1 & 0\\0&0},\qquad M_\gamma=\mat{1&0},\qquad M_\delta=\mat{0\\1}.
    \]
    Then $M$ has submodules $A,B$ with $\undim A=(1,0,1)$ and $\undim B=(1,1,1)$. It also has quotient modules $A',B'$ with the same dimension vectors as $A,B$ but they are not isomorphic. Since submodules and quotient modules are not isomorphic, $M$ is a brick. But $\cD_{KQ}(M)$ is only $1$-dimensional, spanned by the vector $(-1,0,1)$. So, $\cD_{KQ}(M)$ is not a wall.
\end{Example}
}

\begin{Definition}\label{def: red and green boundary points}
    A point $v$ in the boundary $\partial \cU$ of an open chamber $\cU$ is called \textbf{green} if $v-\varepsilon \eta\in\cU$ for all sufficiently small $\varepsilon>0$ where we recall that $\eta=(1,1,\cdots,1)$. Similarly $w\in\partial\cU$ is called a \textbf{red boundary point} of $\cU$ if $v+\varepsilon\eta\in\cU$ for all sufficiently small $\varepsilon>0$. (See Figure \ref{Fig: hour glass} where $v_0$ is a green boundary point of $\cB_-$ and a red boundary point of $\cB_+$.)
\end{Definition}

{
\begin{Remark}\label{rem: red/green walls}
    Note that $f_\varepsilon(v)=v-\varepsilon\eta$ is a continuous function $f_\varepsilon:\partial U\to \RR^n$. Therefore, the set of green boundary points of an open chamber $\cU$ is a relatively open subset of the boundary. Indeed, $f_\varepsilon^{-1}(\cU)$ is an open neighborhood $N$ of $v$ in $\partial \cU$. For any $w\in N$, the line $\{w+t\eta)$, $t<\varepsilon$ will meet $\partial \cU$ at only one point, $w$ since $\cU$ is convex. So, every point in $N$ is a green boundary point of $\cU$.

    In the case of the standard open chamber $\cC\left(\symm AB\right)$, the boundary is the union of $n$ walls. Every point $v$ in the interior of each wall is either a red or green boundary point since the chamber is on one side of the wall or the other. If $v\in\partial \cU$ is not in the interior of any wall, it lies on the boundary of at least two walls. If $v$ is a green boundary point, all points in a neighborhood of $v$ in $\partial \cU$ are green. Therefore, all walls containing $v$ must be green. Conversely, suppose all the walls containing $v$ are green. Then we claim that $v$ is also green. To see this, take two points $x,y\in\partial \cU$ so that the line connecting $x,y$ contains a point $v+\varepsilon\eta$ for some $\varepsilon$. Then, $x+\delta \eta$ and $y+\delta\eta$ lie in $\cU$ for $\delta>0$ small enough. Since $\cU$ is convex, this implies that $c+\varepsilon+\delta\in \cU$. So, $v$ is a green boundary point of $\cU$.
    \end{Remark}
}

\begin{Definition}\cite{Br}\cite{BST}\label{def: P(v)}
    For any $v\in\RR^n$ there are torsion classes $\cT(v)$ and $\overline\cT(v)$ defined as follows.
    
        \item[(1)] $\cT(v)=\{0\}\cup \cT_0(v)$ where $\cT_0(v)=$
        \[
        \{M\in mod\text-\Lambda\mid v\ast\undim M>0 \text{ and } v\ast\undim M''>0 \text{ for quotients $M''\neq0$ of $M$}\}\]
        
        \item[(2)] $\overline\cT(v)=\{M\in mod\text-\Lambda\mid v\ast\undim M\ge0 \text{ and } v\ast \undim M''\ge0\text{ for quotients $M''$ of $M$.} \}.$
\end{Definition}

It is not too hard to show that $\cT(v)$, $\overline\cT(v)$ are torsion classes \cite[Lemma 6.6]{Br}. Also Bridgeland observes the following. See also \cite[Theorem 2.28]{IM}.

\begin{Proposition}\label{prop: T(v), W(v), Tbar(v)}
For any $v\in\cD_\Lambda$, $(\cT(v),\cW(v))$ form a torsion pair in $\overline\cT(v)$. In particular, $\cT(v),\cW(v)$ are disjoint subsets of $\overline\cT(v)$.
\end{Proposition}

The following is a useful lemma which allows us to determine the set of points $v$ which give the same torsion class $\cT(v)$. 

\begin{Lemma}\cite[Theorem 2.23]{IM}\label{lem: P(v)=P(w)}
    Let $v,w\in\RR^n$ then $\cT(v)=\cT(w)$ if and only if the straight line $\gamma(t)=(1-t)v+tw$ from $v$ to $w$ satisfies the following conditions.
    \begin{enumerate}
\item The path does not cross any domain transversely, i.e., if $\gamma(t)\in\cD_\Lambda(M)$ for some $0<t<1$ and some $M$, then $
    \frac{d}{dt}\gamma(t)\ast \undim M|_{t=0}=0$.
\item If $\gamma(0)\in \cD_\Lambda(M)$ then 
$
    \frac{d}{dt}\gamma(t)\ast \undim M|_{t=0}\le0$.
\item If $\gamma(1)\in \cD_\Lambda(N)$ then $
    \frac{d}{dt}\gamma(t)\ast \undim N|_{t=1}\ge0$.
\end{enumerate}
\end{Lemma}
This lemma implies the following.
\begin{Theorem}\label{thm: set of v with T(v)=T is convex}
    For any torsion class $\cT_0$, the set of all $v\in\RR^n$ so that $\cT(v)=\cT_0$ is convex or empty.
\end{Theorem}

\begin{Remark}\label{rem: too many torsion classes}
Examples of torsion classes which are not of the form $\cT(v)$.

    a) A general example: Let $H$ be a tame hereditary algebra over $\CC$. For each subset $S$ of $\CC P^1$, let $\cT_S$ be the torsion class consisting of all preinjective modules and the homogeneous modules $M_\lambda$ (and their self-extensions) with parameters $\lambda\in S$. The set of such torsion classes has cardinality greater than the continuum. So, the set of all $\cT(v)$ for $v\in \RR^n$ cannot contain all $\cT_S$.

    b) A concrete example: Let $H=KQ$ with $    \xymatrixrowsep{10pt}\xymatrixcolsep{20pt}
\xymatrix{
Q:1\ar@/_.3pc/[r] \ar@/^.3pc/[r] & 2
	} 
$
 be the Kroneker over any field $K$. Let $S$ be a subset of the projective line $KP^1=K\coprod\{\infty\}$ so that $S$ and its complement $S^c$ are nonempty. Then we claim that $\cT_S$, defined in (a) is not equal to $\cT(v)$ for any $v\in \RR^2$. To see this, suppose not. Let $\cT_S=\cT(v)$ where $v=(x,y)$. Then, for all $\lambda\in S$, $v\ast \undim M_\lambda>0$, i.e., $x+y>0$ and $v\ast\undim X>0$ for all quotient modules $X$ of $M_\lambda$. But $X=S_1$ is the only quotient of $M_\lambda$. So, the condition is $x>0$. Since the conditions $x+y>0$ and $x>0$ is independent of $\lambda$ we see that $\cT(v)$ includes $M_\lambda$ for all $\lambda\in KP^1$, a contradiction. So, $\cT(v)\neq \cT_S$.
\end{Remark}


\begin{Theorem}\cite{Br}\cite[Theorems 2.20, 2.21]{IM}\label{thm: chambers are convex}
    Let $\cU$ be a chamber. Then $\cU$ is convex and the torsion class $\cT(v)$ is constant for all $v\in\cU$. So, we can define $\cT(\cU)$ to be $\cT(v)$ for any $v\in \cU$. Furthermore, $\cT(\cU)\neq\cT(\cU')$ for $\cU\neq \cU'$.
\end{Theorem}

{
Recall Theorem \ref{thm: bijection T - AB - C(AB)}. A complete balanced pair $\symm AB$, is associated to the functorially finite torsion class $\Gen\,A$ and all functorially finite torsion classes are given in this way. The associated chamber is $\cC\left(\symm AB\right)$, the standard open chamber associated to $\symm AB$ by Def. \ref{def: standard open chamber}. See Figure \ref{fig: stardard chamber}. The following theorem holds for any balanced pair.
}

\begin{Theorem}\label{thm: P(U) is standard}
{Let $\symm AB$ be a balance pair, not necessarily complete. Then
    $$
\cT\!\left(    g\!\left(\symm AB\right)\right)=\Gen\,A.$$
}
\end{Theorem}

\begin{proof}

    Take any $X\in \Gen\,A$. Then $\Hom_\Lambda(X,B)=0$. By Corollary \ref{cor: g(A/B) dot X}, we have
    \[
        g\!\left(\symm AB\right)\ast \undim X=\dim_K \Hom_\Lambda(A,X)-\dim_K\Hom_\Lambda(X,B)>0.
    \] Since this also holds for any nonzero quotient of $X$, we have $X\in\cT\left(g\!\left(\tsymm AB\right)\right)$. So, $\Gen\,A\subseteq\cT\left(g\!\left(\tsymm AB\right)\right)$.

    Conversely, suppose $Y\in \cT\left(g\!\left(\tsymm AB\right)\right)$ which is not in $\Gen\,A$. Take $Y$ to be of minimal length with this condition. As an element of $\cT\left(g\!\left(\tsymm AB\right)\right)$ we have:
    \[
    \dim_K\Hom_\Lambda(A,Y)>\dim_K\Hom_\Lambda(Y,B).
    \]
    So, there is a nonzero morphism $f:A\to Y$. Let $Y'$ be the image of $f$ and $Y''$ the cokernel. Then we have an exact sequence
    \[
        0\to Y'\to Y\to Y''\to 0.
    \]
        Since $Y''$ is a quotient of $Y$ we have $Y''\in \cT\left(g\!\left(\tsymm AB\right)\right)$. By induction on length we have $Y''\in \Gen\,A$. Since $\Gen A$ is a torsion class by Proposition \ref{prop: Gen M is torsion class}, we have $Y\in\Gen A$. So, $\cT\left(g\!\left(\tsymm AB\right)\right)=\Gen\,A$.
\end{proof}

\begin{Lemma}\cite[Lemma 2.10]{IM}\label{lem: P=Pbar}
    We have $\cT(v)=\overline\cT(v)$ if and only if $v\notin\cD_\Lambda$. In particular $\cT(v)=\overline\cT(v)$ for $v$ in any chamber $\cU$.
\end{Lemma}

Thus, for $v$ in the boundary of an open chamber $\cU$, $\cT(v)\subsetneq \overline\cT(v)$. 

\begin{Lemma}\cite[Theorem 2.24]{IM}\label{lem: Pbar(v)=Pbar(w)}
    Let $v,w\in\RR^n$ then $\overline\cT(v)=\overline\cT(w)$ if and only if the straight line $\gamma(t)=(1-t)v+tw$ from $v$ to $w$ satisfies the following conditions.
    \begin{enumerate}
\item The path does not cross any domain transversely, i.e., if $\gamma(t)\in\cD_\Lambda(M)$ for some $0<t<1$ and some $M$, then $
    \frac{d}{dt}\gamma(t)\ast \undim M|_{t=0}=0$.
\item If $\gamma(0)\in \cD_\Lambda(M)$ then 
$
    \frac{d}{dt}\gamma(t)\ast \undim M|_{t=0}\ge0$.
\item If $\gamma(1)\in \cD_\Lambda(N)$ then $
    \frac{d}{dt}\gamma(t)\ast \undim N|_{t=1}\le0$.
\end{enumerate}
\end{Lemma}

\begin{Proposition}\label{prop: P(v) for green v}
    Let $\cU$ be any open chamber and let $v\in\partial \cU$.
    \begin{enumerate}
        \item[(a)] If $v$ is a green boundary point of $\cU$, i.e., $v-\varepsilon\eta\in\cU$ for small $\varepsilon>0$, then $\cT(v)=\cT(\cU)$.
        \item[(b)] If $v$ is a red boundary point of $\cU$, i.e., $v+\varepsilon\eta\in\cU$ for small $\varepsilon>0$, then $\overline\cT(v)=\overline\cT(\cU)=\cT(\cU)$.
        \item[(c)] In all cases, $\cT(v)\subseteq\cT(\cU)\subseteq \overline\cT(v)$.
    \end{enumerate}
    \end{Proposition}

\begin{proof} (a) The boundary point $v$ will lie in at least one domain, say $\cD_\Lambda(M_\alpha)$.
    Let $\gamma(t)=v-t\varepsilon\eta$ for some small $\varepsilon>0$. Since this path is in the interior of $\cU$ for $0<t\le1$, it does not cross any wall when $t\in (0,1]$. When $t=0$ we have
    \[
    \frac d{dt} \gamma(t)\ast \cD_\Lambda(M_\alpha)|_{t=0}=-\varepsilon \eta\ast \undim M_\alpha<0
    \]
    for all modules $M_\alpha$ for which $v\in \cD_\Lambda(M_\alpha)$. By Lemma \ref{lem: P(v)=P(w)}, we have
\[
\cT(v)=\cT(v-\varepsilon \eta)=\cT(\cU).\]

(b) This has a similar proof as (a) using Lemma \ref{lem: Pbar(v)=Pbar(w)} in place of Lemma \ref{lem: P(v)=P(w)}.

(c) To see that, in general, $\cT(v)\subseteq\cT(\cU)$, let $X\in\cT(v)$. Then $v\ast\undim X''>0$ for all nonzero quotients $X''$ of $X$. But there are only finitely many possibilities for the vector $\undim X''$ for fixed $X$, so for $w\in\cU$ sufficiently close to $v$, $w\ast\undim X''>0$ for all nonzero quotients $X''$ of $X$. So, $\cT(v)\subseteq\cT(\cU)$.

To see that $\cT(\cU)\subseteq \overline\cT(v)$, let $X\notin \overline\cT(v)$. Then there exists a nonzero quotient $X''$ of $X$ so that $v\ast\undim X''<0$. Then $w\ast\undim X''<0$ for $w\in\cU$ sufficiently close to $v$. So, $X\notin \overline\cT(\cU)=\cT(\cU)$. Therefore, $\cT(v)\subseteq\cT(\cU)\subseteq \overline\cT(v)$ for all $v\in\partial \cU$.
\end{proof}

{
We need the following refinement of Proposition \ref{prop: P(v) for green v}.
\begin{Theorem}\label{thm: not red or green}
    Let $\cU$ be a standard open chamber $\cU=\cC\left(\symm AB\right)$ for a complete balanced pair $\symm AB=\bigsqcup \symm{A_i}{B_i}$. Let $v\in\partial \cU$. 
        \begin{enumerate}
        \item[(a)] The point $v$ is a green boundary point of $\cU$ if and only if $\cT(v)=\cT(\cU)$.
        \item[(b)] The point $v$ is a red boundary point of $\cU$ if and only if $\overline\cT(v)=\overline\cT(\cU)=\cT(\cU)$.
        \item[(c)] If $v$ is neither red nor green, then $\cT(v)\subsetneq\cT(\cU)\subsetneq \overline\cT(v)$.
    \end{enumerate}
\end{Theorem}

\begin{proof}
The simplicial cone $\cU=\cC\left(\symm AB\right)$ has $n$ walls where $n$ is the rank of $\Lambda$, the number of simple $\Lambda$-modules. Each of these walls has the property that its interior points are either all red or all green (Remark \ref{rem: red/green walls}). Also, by Remark \ref{rem: red/green walls}, $v\in\partial\cU$ is green (or red) if and only if every adjacent wall is green (or red) where ``adjacent'' means it contains $v$. In case (c) when $v\in\partial\cU$ is neither red nor green, there must be an adjacent red wall $\cD_\Lambda(M_1)$ and an adjacent green wall $\cD_\Lambda(M_2)$. See Figure \ref{Fig: not red/green}.

Take points $v_1\in\cD_\Lambda(M_1)$ and $v_2\in\cD_\Lambda(M_2)$. Since $v_1$ is red, $\overline\cT(v_1)=\cT(\cU)$ by Proposition \ref{prop: P(v) for green v}. Using Proposition \ref{prop: T(v), W(v), Tbar(v)}, we have $M_1\in \cW(v_1)\subset \overline\cT(v_1)=\cT(\cU)$. But $M_1\in \cW(v)$. So, $M_1\notin \cT(v)$. So, $\cT(v)\neq\cT(\cU)$. 

Now look at $v_2\in\cD_\Lambda(M_2)$. Since $v_2$ is green, $\cT(v_2)=\cT(\cU)$. But $M_2\notin\cT(v_2)=\cT(\cU)$. And $M_2\in\cW(v)\subset\overline \cT(v)$. So, $\overline\cT(v)\neq \cT(\cU)$. Therefore, 
\[
	\cT(v)\subsetneq\cT(\cU)\subsetneq \overline \cT(v)
\]
as claimed. (a) and (b) follow from (c) and Proposition \ref{prop: P(v) for green v}.
\end{proof}

}

{
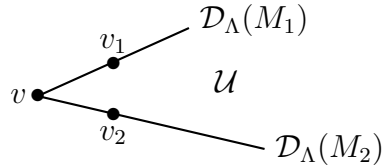
\begin{figure}[htbp]
\begin{center}
\begin{tikzpicture}

\draw[thick] (3,-.7)--(0,0)--(2,.9);
\draw (2.5,0.2) node{$\cU$};
\draw (0,0) node{$\bullet$};
\draw (1,0.43) node{$\bullet$};
\draw (1,-.25) node{$\bullet$};
\draw (0,0) node[left]{$v$};
\draw (1,-.25) node[below]{$v_2$};
\draw (2,1) node[right]{$\cD_\Lambda(M_1)$};
\draw (3,-.7) node[right]{$\cD_\Lambda(M_2)$};
\draw (1,0.43) node[above]{$v_1$};

\end{tikzpicture}
\caption{In the proof of Theorem \ref{thm: not red or green}, $\cU= \cC\left(\symm AB\right) $ and $v\in\partial\cU$ is neither red nor green. Then, there is an adjacent red wall $\cD_\Lambda(M_1)$ and an adjacent green wall $\cD_\Lambda(M_2)$ at $v$ and red and green points $v_1,v_2$ on those walls. ``Adjacent'' means $v\in \cD_\Lambda(M_1)$ and $v\in \cD_\Lambda(M_2)$ as shown.}
\label{Fig: not red/green}
\end{center}
\end{figure}
}

\begin{Theorem}\label{thm: Pbar of A/B}
    Let $\symm AB$ be a balanced pair, not necessarily complete. Then
\[
\overline\cT\left(g\!\left(\symm AB\right)\right)=\,^\perp B.
    \]
\end{Theorem}

\begin{proof}
    Let $X\in\,^\perp B$. Then $X''\in\,^\perp B$ for any quotient $X''$ of $X$. So, by Corollary \ref{cor: g(A/B) dot X},
    \[
   g\!\left(\symm AB\right) \ast \undim X=\dim_K\Hom_\Lambda(A,X)-\dim_K\Hom_\Lambda(X,B)=\dim_K\Hom_\Lambda(A,X)\ge0,
    \]
\[
   g\!\left(\symm AB\right) \ast \undim X''=\dim_K\Hom_\Lambda(A,X'')-\dim_K\Hom_\Lambda(X'',B)=\dim_K\Hom_\Lambda(A,X'')\ge0.
    \]
    This implies $X\in\overline\cT\left(g\!\left(\tsymm AB\right)\right)$. So, $\,^\perp B\subseteq \overline\cT\left(g\!\left(\tsymm AB\right)\right)$.

    Conversely, let $Y\in \overline\cT\left(g\!\left(\tsymm AB\right)\right)$. To show $Y\in \,^\perp B$, suppose not. Then, there is a nonzero morphism $f:Y\to B$. Then the image of $f$, $Y''$ is a quotient of $Y$ and, therefore, $Y''\in \overline\cT\left(g\!\left(\tsymm AB\right)\right)$. Also $Y''\subseteq B$. But then $\Hom_\Lambda(A,Y'')=0$. So,
    \[
          g\!\left(\symm AB\right) \ast \undim Y''=\dim_K\Hom_\Lambda(A,Y'')-\dim_K\Hom_\Lambda(Y'',B)=-\dim_K\Hom_\Lambda(Y'',B)<0.
    \]
    So, $Y''\notin \overline\cT\left(g\!\left(\tsymm AB\right)\right)$. This contradiction shows that $\overline\cT\left(g\!\left(\tsymm AB\right)\right)\subseteq \,^\perp B$. So, they are equal.
\end{proof}

}

{
\section{Covering and flooring}\label{sec: cover and floor}

We analyze two particular chambers associated to each balanced pair called the ``upper chamber'' and the ``lower chamber''. We describe these chambers when the balanced pair is not complete and show that such chambers do not exist for complete balanced pairs. 


{
\begin{figure}[htbp]
\begin{center}
\begin{tikzpicture}
\draw[very thick] (0,0)..controls (1,1) and (3,2)..(4,2);
\draw[very thick] (0,0)..controls (-1,1) and (-3,2)..(-4,2);
\draw[very thick] (0,0)..controls (-1,-1) and (-2,-2)..(-2.5,-4);
\draw[very thick] (0,0)..controls (1,-1) and (2,-2)..(2.5,-4);
\draw[very thick] (-4,1)..controls (-3,2) and (-1.4,3)..(0,3);
\draw[very thick] (4,1)..controls (3,2) and (1.4,3)..(0,3);
\draw[very thick] (-3.5,-4)..controls (-2.5,-3.5) and (-1.5,-3).. (0,-3); 
\draw[very thick] (3.5,-4)..controls (2.5,-3.5) and (1.5,-3).. (0,-3); 
\draw[thick] (-4,0)..controls (-3,.5) and (-1,.3)..(0,0);
\draw[thick] (-4,-1)..controls (-3,-.5) and (-1,0)..(0,0);
\draw[thick] (-4,-1.5)..controls (-3,-1) and (-1,-.1)..(0,0);
\draw[thick] (-4,-2)..controls (-3,-1.5) and (-1,-.2)..(0,0);
\draw[thick] (4,-2)..controls (3,-1.5) and (1,-.4)..(0,0);
%
\draw[red](0,2.6) node{\large$\cD_\Lambda(M_-)$};
\draw[black!45!green] (0,-2.7) node{\large$\cD_\Lambda(M_+)$};
\draw (-4.75,0) node{$\cD_\Lambda(X)$};
\draw (-4.75,-1) node{$\cD_\Lambda(Y)$};
\draw (-4.75,-1.5) node{$\cdots$};
\draw (-4.75,-2) node{$\cD_\Lambda(Z)$};
\draw (4.75,-2) node{$\cD_\Lambda(W)$};
\draw (2.5,1.1) node {$\cD_\Lambda(M_1)$};
\draw (-2,.8) node {$\cD_\Lambda(M_{n-1})$};
\draw (1.7,-2.2) node[right] {$\cD_\Lambda(M_{n-1})$};
\draw (-1.7,-2.2) node[left] {$\cD_\Lambda(M_{1})$};
\draw[dashed] (0,-.9)--(0,.9);
\draw (0,.7)node[above]{\small$v_0-\varepsilon\eta$};
\draw (0,-.7)node[below]{\small$v_0+\varepsilon\eta$};
\draw (0.1,0) node[right]{$v_0=g\!\left(\symm {A_0}{B_0}\right)$};
\draw (-3,1.9) node[above]{$v_1$};
\draw (3,1.9) node[above]{$v_{n-1}$};
\draw (2.6,-3.3) node{$v_1'$};
\draw (-1.85,-3.65) node{$v_{n-1}'$};
\draw (0,1.8) node{$\mathcal B_-$};
\draw (0,-1.8) node{$\mathcal B_+$};
\draw[thick,green!80!black,->] (1,4)--(1,-4);
\draw[green!80!black] (1,4) node[right]{$\gamma$};
{
\begin{scope}[yshift=-.5mm]
\clip (-3,1.3)rectangle(3,3.1);
\draw[ thick,red] (-4,1)..controls (-3,2) and (-1.4,3)..(0,3);
\draw[thick,red] (4,1)..controls (3,2) and (1.4,3)..(0,3);
\end{scope}
\begin{scope}[yshift=.5mm]
\clip (-2.3,-4)rectangle (2.3,-2.7);
\draw[thick,black!45!green] (-3.5,-4)..controls (-2.5,-3.5) and (-1.5,-3).. (0,-3); 
\draw[thick,black!45!green] (3.5,-4)..controls (2.5,-3.5) and (1.5,-3).. (0,-3); 
\end{scope}
}
\end{tikzpicture}
\caption{Here $v_0=g(\tsymm {A_0}{B_0})$ where $A_0$ and $B_0$ are nonzero and indecomposable. The wide subcategory $\cW(g(\tsymm {A_0}{B_0}))$ has $n-1$ simple objects $M_1,\cdots, M_{n-1}$ and composite objects $X,Y,Z$, etc. We will show that, in $S^{n-1}$, the projected domains $\cD_\Lambda(M_i)$ form $n-1$ of the $n$ number of faces of the two projected $(n-1)$-simplices $\Delta_-$ and $\Delta_+$ with interiors $\cB_-$ and $\cB_+$ as shown. This figure shows the projection of these domains and chambers into the unit sphere $S^{n-1}$. The vertical green arrow indicates a linear green path $\gamma$ going through $\cB_-$ and $\cB_+$, $\cB_-$ is the upper chamber containing $v_0-\varepsilon\eta$ and $\cB_+$ is the lower chamber containing $v_0+\varepsilon\eta$.}
\label{Fig: hour glass}
\end{center}
\end{figure}
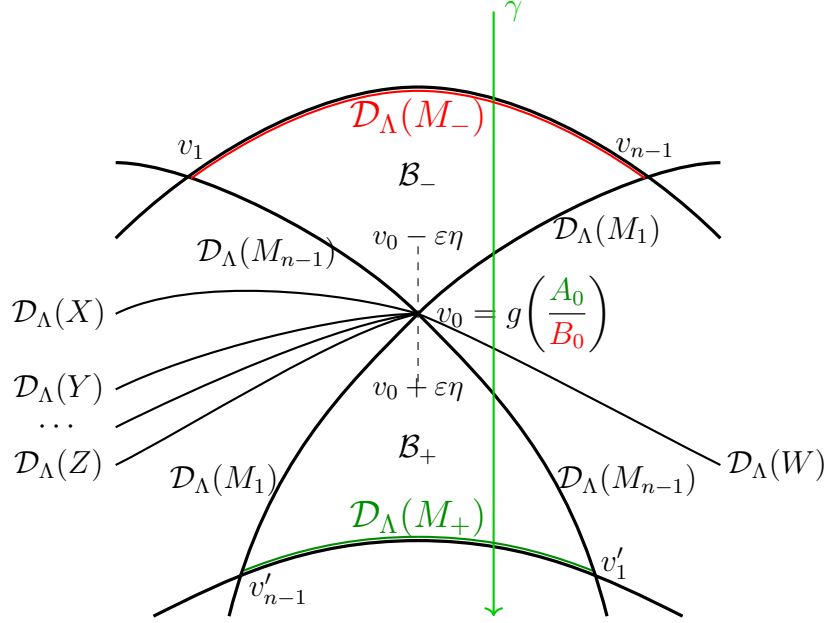
}

{\begin{Definition}\label{def: chamber A}
Let $\symm {A}{B}$ be a non-complete balanced pair and its $g$-vector $g\!\left(\symm{A}{B}\right)$. We denote by $\cB_-$ the chamber having the property that $g\!\left(\symm{A}{B}\right)\in\partial \cB_-$ is a green boundary point of $\cB_-$ (Def. \ref{def: red and green boundary points}). We call $\cB_-$ the \textbf{upper chamber} of $g(\tsymm{A}{B})$ if it exists. We also say that $\cB_-$ is on the \textbf{negative side} of $g(\tsymm{A}{B})$.
\end{Definition}

\begin{Proposition}\label{prop: uniqueness of B-}
\item[(a)] The upper chamber $\cB_-$ of $g\!\left(\symm{A}{B}\right)$, the $g$-vector of a non-complete balanced pair, is unique if it exists.

\item[(b)] Suppose $\symm{A_1}{B_1},\cdots,\symm{A_k}{B_k}$, $k<n$, are some, but not all, of the indecomposable summands of a complete balanced pair $\symm AB$. Let $\cU=\cC\left(\symm AB\right)$ be the corresponding standard open chamber. Suppose $\cT(\cU)=\cT\left(g\!\left(\bigsqcup \symm{A_i}{B_i}\right)\right)=\cT\left(\sum g\!\left(\symm{A_i}{B_i}\right)\right)$. Then $g\!\left(\bigsqcup \symm{A_i}{B_i}\right)$ is a green boundary point of $\cU$. So, $\cU=\cC\left(\symm AB\right)=\cB_-$.
\end{Proposition}

\begin{proof}
(a) We have $\cT(\cB_-)=\cT\left(g\!\left(\bigsqcup \symm{A_i}{B_i}\right)\right)$ by Proposition \ref{prop: P(v) for green v}. This determines $\cB_-$ uniquely if it exists by Theorem \ref{thm: chambers are convex}.

(b) Follows from Lemma \ref{lem: P(v)=P(w)} and Theorem \ref{thm: not red or green}. If $\cT\left(g\!\left(\bigsqcup \symm{A_i}{B_i}\right)\right)=\cT(\cU)$, then, by Lemma \ref{lem: P(v)=P(w)}, $g\!\left(\bigsqcup \symm{A_i}{B_i}\right)$ must be a boundary point of $\cU$. Since $\cU$ is standard, Theorem \ref{thm: not red or green} applies and tells us that $g\!\left(\bigsqcup \symm{A_i}{B_i}\right)$ is a green boundary point. Therefore, $\cU=\cC(\tsymm AB)=\cB_-$.
\end{proof}
}

{\begin{Definition}\label{def: chamber B}
Let $\symm {A}{B}$ be a non-complete balanced pair and its $g$-vector $g\!\left(\symm{A}{B}\right)$. We denote by $\cB_+$ the chamber having the property that $g\!\left(\symm{A}{B}\right)\in\partial \cB_+$ is a red boundary point of $\cB_+$ (Def. \ref{def: red and green boundary points}). We call $\cB_+$ the \textbf{lower chamber} of $g(\tsymm{A}{B})$ if it exists. We also say that $\cB_+$ is on the \textbf{positive side} of $g(\tsymm{A}{B})$.
\end{Definition}

\begin{Proposition}\label{prop: uniqueness of B+}
\begin{enumerate}
\item[(a)] The lower chamber $\cB_+$ of $g\!\left(\symm{A}{B}\right)$, the $g$-vector of a non-complete balanced pair, is unique if it exists.

\item[(b)] Suppose $\symm{A_1}{B_1},\cdots,\symm{A_k}{B_k}$, $k<n$, are some, but not all, of the indecomposable summands of a complete balanced pair $\symm AB$. Let $\cU=\cC\left(\symm AB\right)$ be the corresponding standard open chamber. Suppose $\cT(\cU)=\overline\cT\left(g\!\left(\bigsqcup \symm{A_i}{B_i}\right)\right)$. Then $g\!\left(\bigsqcup \symm{A_i}{B_i}\right)$ is a red boundary point of $\cU$. So, $\cU=\cC\left(\symm AB\right)=\cB_+$.
\end{enumerate}
\end{Proposition}

\begin{proof} The proof is similar to the proof of Proposition \ref{prop: uniqueness of B-} using Lemma \ref{lem: Pbar(v)=Pbar(w)} instead of Lemma \ref{lem: P(v)=P(w)}.
\end{proof}
}

By Theorem \ref{thm: the open simplex}, the objects $M$ of the wide subcategory $\cW(g(\tsymm{A_0}{B_0}))$ for an indecomposable balanced pair $\tsymm{A_0}{B_0}$ label the walls $\cD_\Lambda(M)$ which contain $g(\tsymm{A_0}{B_0})$. By Theorem \ref{thm: J perpendicular}, this wide subcategory is generated by $n-1$ simple objects $M_1,\cdots,M_{n-1}$. We will show that these form part of the boundary of both the upper chamber $\cB_-$ on the negative side of $g(\tsymm {A_0}{B_0})$ and the lower chamber $\cB_+$ on the positive side of $g(\tsymm {A_0}{B_0})$. See Figure \ref{Fig: hour glass}. 

{
 Using results from the last section we can determine the torsion classes associated to the chambers $\cB_-$ and $\cB_+$ if they exist. This hypothetical statement will be used later to prove the existence of $\cB_-$ and $\cB_+$.

\begin{Theorem}\label{thm: P(U0),P(U+)}
Let $\symm {A}{B}$ be a non-complete balanced pair. Let $\cB_-$, $\cB_+$ be the open chambers on the negative and positive sides of $g\!\left(\symm {A}{B}\right)$, resp. if they exist. Then, the torsion classes associated to $\cB_-$ and $\cB_+$ are given as follows.
\[
    \cT(\cB_-)=\cT\!\left(g\!\left(\symm {A}{B}\right)\right)=\Gen\,A.
\]
\[
    \cT(\cB_+)=\overline\cT(\cB_+)=\overline\cT\!\left(g\!\left(\symm {A}{B}\right)\right)=\,^\perp B.
\]
\end{Theorem}

\begin{proof}
    Since $\cB_-$ is on the negative side of $g(\tsymm {A}{B})$, $g(\tsymm {A}{B})$ is a ``green'' boundary point of $\cB_-$. By Proposition \ref{prop: P(v) for green v}, $\cT(\cB_-)=\cT(g(\tsymm {A}{B}))$. By Theorem \ref{thm: P(U) is standard}, $\cT(g(\tsymm {A}{B}))=\Gen\,A$.

    Since $\cB_+$ is on the positive side of $g(\tsymm {A}{B})$, $g(\tsymm {A}{B})$ is a ``red'' boundary point of $\cB_+$. By Proposition \ref{prop: P(v) for green v}, $\overline\cT(\cB_+)=\overline \cT(g(\tsymm {A}{B}))$. By Theorem \ref{thm: Pbar of A/B}, $\overline \cT(g(\tsymm {A}{B}))=\,^\perp B$. By Lemma \ref{lem: P=Pbar}, $\cT(\cB_+)=\overline\cT(\cB_+)$. Therefore, $\cT(\cB_+)=\,^\perp B$.
\end{proof}
}

{
\begin{Remark}\label{rem: B-B+ do not exist}
Let $\symm AB$ be a complete balanced pair. Then $g\!\left(\symm AB\right)$ is not on the boundary of any chamber since this is an interior of the standard open chamber $\cC\left(\symm AB\right)$. Therefore, the upper and lower chambers of $g\!\left(\symm AB\right)$ do not exist. However, in all other cases we will show that they exist.
\end{Remark}
}

\begin{Remark}\label{eg: negative chamber}
    Suppose that $A_0=0$. Then: 
    
    (a) The indecomposable balanced pair is $\symm 0{I_j}$ with $g$-vector the negative unit vector $g(\tsymm 0{I_j})=-e_j$. The wide subcategory of this vector is
    \[
       \cW(-e_j)= J\!\left(\symm0{I_j}\right)=\,^\perp I_j
    \]
    which is the category of all modules $M$ not having the $j$th simple module $S_j$ in its composition series. 
    
    (b) The simple objects of this category are the simple modules $S_i$ for $i\neq j$. 
    
(c) Let $\symm AB$ be the complete balanced pair with $A=0$ and $B=\bigsqcup I_i$, the sum of all injective modules. The corresponding standard open chamber is $\cC\left(\symm AB\right)$. Then
\[
    \cT\left(\cC\left(\symm AB\right)\right)=\Gen A=0=\Gen A_0.
\]
By Proposition \ref{prop: uniqueness of B-}, $\cC\left(\symm AB\right)=\cB_-$ is the upper chamber of $g(\tsymm 0{I_j})$.

 (d) This open chamber is the set of all positive linear combinations of $g(\tsymm 0{I_i})=-e_i$, i.e., $\cB_-$ is the set of all vectors in $\RR^n$ all of whose coordinates are negative. We call this the \textbf{negative chamber} and write $\cB_-=\cU_{neg}$.
\end{Remark}

\begin{Remark}\label{eg: positive chamber}
    Suppose $B_0=0$. Then $A_0$ is projective and $\symm {A_0}{B_0}$ is a component of the complete balanced pair $\symm AB$ where $B=0$ and $A=\bigsqcup P_i$ is the sum of all projective modules. Then
    \[
\overline\cT\left(g\!\left(\symm{A_0}{B_0}\right)\right)=\,^\perp B_0=mod\text-\Lambda=\,^\perp B=\cT\left(\cC\left(\symm AB\right)\right).
    \]
    Therefore, by Proposition \ref{prop: uniqueness of B+}, $\cC\left(\symm AB\right)=\cB_+$ is the lower chamber of $g\!\left(\symm {A_0}{B_0}\right)$.

This open chamber is the set of all positive linear combinations of the $g$-vectors $g(\tsymm {P_i}0)=e_i$, i.e., $\cB_+$ is the set of all vectors in $\RR^n$ all of whose coordinates are positive. We call this the \textbf{positive chamber} and write $\cB_+=\cU_{pos}$.
 \end{Remark}

\begin{Definition}\label{def: green path}
A \textbf{green path} for $mod\text-\Lambda$ is a smooth path $\gamma(t)$ in $\RR^n$, $t\in\RR$, satisfying the following conditions.
\begin{enumerate}
\item The path $\gamma(t)$ has all negative coordinates for $t<<0$.
\item The path $\gamma(t)$ has all positive coordinates for $t>>0$.
\item  When $\gamma(t_0)\in \cD_\Lambda(M)$, its velocity is in the ``positive direction'', i.e.,
\[
	\left.\frac{d}{d t} \gamma(t)\right|_{t_0}\!\ast\undim M>0 .
\]
\end{enumerate}
\end{Definition}

A simple example of a green path is the \textbf{linear green path} given by
\[
    \gamma(t)=v+t\eta
\]
where $v\in\RR^n$ and $\eta=(1,1,\cdots,1)$. Then (1) and (2) are clearly satisfied and 
\[
\left.\frac{d}{d t} \gamma(t)\right|_{t_0}\!\ast\undim M=\eta \ast\undim M=\dim_K M
\]
is positive for nonzero $M$.

\begin{Proposition}\label{prop: there is one positive and one negative domain}
    \emph{(a)} Let $\cU\neq\cU_{neg}$ be an open chamber. Then
    \begin{enumerate}
        \item The chamber $\cU$ is on the positive side of at least one of the walls on its boundary.
        \item Let $v\in \cU$. Then the linear green path $\gamma(t)=v+t\eta$ passes through at least one of these walls for some $t<0$.
    \end{enumerate}

    \emph{(b)} Let $\cU\neq\cU_{pos}$ be an open chamber. Then $\cU$ is on the negative side of at least one of the walls on its boundary. Furthermore, for any $v\in \cU$, the linear green path $\gamma(t)=v+t\eta$ passes through at least one of these walls for some $t>0$.
\end{Proposition}

\begin{proof}
    (a) Let $v\in \cU$. Since $\cU\neq\cU_{neg}$, the linear green path $\gamma(t)=v+t\eta$ which starts in $\cU_{neg}$ for $t<<0$ and stops at $v\in \cU$ when $t=0$ must intersect the boundary of $\cU$ at some $t_0<0$. Since $\partial \cU\cap S^{n-1}$ is compact, we can take $t_0$ to be maximal. Then $\gamma(t_0)\in \cD_\Lambda(M)$ for at least one $M$ and $\cU$ is on the positive side of $\cD_\Lambda(M)$, as claimed. The proof of (b) is similar. 
\end{proof}

A well-known useful property of green paths is that it gives a ``forward hom-orthogonal sequence''. This holds over $mod\text-\Lambda$ and it was extended to green paths in the stability diagram of any torsion class $\cG\subseteq mod\text-\Lambda$ in \cite[Theorem 3.4]{IM}. The statement is as follows.

\begin{Theorem}\label{thm: FHO seq}
Let $\gamma(t)$ be a green path for $mod\text-\Lambda$. Suppose $\gamma(t_0)\in \cD_\Lambda(M_0)$ and $\gamma(t_1)\in \cD_\Lambda(M_1)$ where $t_0<t_1$. Then $\Hom_\Lambda(M_0,M_1)=0$.
\end{Theorem}

\begin{Corollary}\label{cor: green path hits simple domains} Let $\tsymm{A_0}{B_0}$ be an indecomposable balanced pair. Let $\gamma(t)$ be a linear green path passing close to $g(\tsymm {A_0}{B_0})$ as shown in Figure \ref{Fig: hour glass}. Then
\begin{enumerate}
    \item[(a)] There exist $t_1<t_2$ and simple objects $M_1,M_2$ in $\cW(g(\tsymm{A_0}{B_0}))$ so that $\gamma(t_1)\in\cD_\Lambda(M_1)$ and $\gamma(t_2)\in\cD_\Lambda(M_2)$.
    \item[(b)] For any $X\in\cW(g(\tsymm{A_0}{B_0}))$, either $\cD_\lambda(X)\cap im(\gamma)=\emptyset$ or $\gamma(t)\in \cD_\Lambda(X)$ for some $t_1\le t\le t_2$. 
\end{enumerate}
\end{Corollary}

\begin{proof}
Since $M_i$ are simple objects of $\cW(g(\tsymm{A_0}{B_0}))$, the walls $\cD_\Lambda(M_i)$ will contain a neighborhood of $g(\tsymm{A_0}{B_0})$ in the hyperplane perpendicular to $\undim M_i$. Therefore, any nearby green path will pass through all of these walls. So $\gamma(t_i)\in \cD_\Lambda(M_i)$ for some $t_i$.

Let $t_0$ be minimal so that $\gamma(t_0)\in \cD_\Lambda(X)$ for some $X$ in $\cW(g(\tsymm {A_0}{B_0}))$. Then $X$ has a filtration where all subquotients are isomorphic to some $M_j$. Let $M_i$ be an object at the top of $X$. I.e., there is an epimorphism $X\onto M_i$. By forward hom-orthogonality we have $t_i\le t_0$. Since $t_0$ is minimal, $t_i=t_0$. So, $\cD_\Lambda(M_i)$ is the first wall crossed by $\gamma$. Similarly, another of these walls $\cD_\Lambda(M_j)$ will be the last wall crossed by $\gamma$.
\end{proof}

{
\begin{Definition}\label{def: Ext-projective}
An object $X$ in a torsion class $\cT$ is called \textbf{Ext-projective} if $\Ext_\Lambda^1(X,Y)=0$ for all $Y\in\cT$.
\end{Definition}
Note that, since $\cT$ is closed under extension, the property of an object being Ext-projective is independent of the larger category that $\cT$ is in.

\begin{Proposition}\label{prop: Ext-proj is tau-rigid}
Let $X$ be an Ext-projective object of a torsion class $\cT$. Then $X$ is $\tau$-rigid.
\end{Proposition}

\begin{proof}
By Lemma \ref{lem: AS}, $X$ being Ext-projective implies $\Hom_\Lambda(Y,\tau X)=0$ for all $Y\in\cT$. In particular $\Hom_\Lambda(X,\tau X)=0$. So, $X$ is $\tau$-rigid.
\end{proof}

Recall that a module $M$ is \textbf{sincere} if it has full support, i.e., $\Hom_\Lambda(M,I)\neq0$ for every nonzero injective module $I$. Similarly, a torsion class $\cT$ is \textbf{sincere} if $\Hom_\Lambda(\cT,I)\neq0$, i.e., $\cT\not\subseteq \,^\perp I$ for any nonzero injective module $I$.

{
\begin{Definition}\label{def: Bongartz}
Let $X$ be a $\tau$-rigid module with $k$ summands. Then, by \cite{AIR}, $X$ can be completed in a unique way to a $\tau$-tilting module $X'=X\sqcup X_1\sqcup\cdots\sqcup X_{n-k}$ where $X_i$ are the Ext-projective objects of the torsion class $\,^\perp \tau X$. Then $X'$ is called the \textbf{Bongartz completion} of $X$. The objects $X_1,\cdots,X_{n-k}$ form the \textbf{Bongartz complement} of $X$.
\end{Definition}

\begin{Definition}\label{def: generalized Bongartz}
Let $\symm AB$ be a balanced pair. Then the \textbf{generalized Bongartz completion} $\symm {A'}{B'}$ of $\symm AB$ is defined as follows.
\begin{enumerate}
\item If all components of $A$ are nonzero, i.e., $\symm 0{I_j}$ is not a summand of $\symm AB$, then $A'$ is the Bongartz completion of $A$ and $B=\tau A$.
\item If $A=0$ then $B$ is injective and $\,^\perp B$ is isomorphic to $mod\text-\Lambda'$ where $\Lambda'=\Lambda/\Lambda e\Lambda$, $e$ being the sum of idempotents corresponding to the components of $B$. Then $A'$ is the sum of the projective $\Lambda'$-modules since the Ext-projective objects of an abelian category are the same as its projective objects. And $B'=B\sqcup \tau A'$.
\item When some of the components of $A$ are zero and some nonzero, say $A_1=\cdots=A_j=0$ and $A_{j+1},\cdots,A_k$ are nonzero where $j< k$, then $B_1,\cdots,B_j$ are injective modules, with direct sum $I$, and $A=A_{j+1}\sqcup\cdots\sqcup A_k$ is a $\tau$-rigid object over $\Lambda'=\Lambda/\Lambda e\Lambda$ where $e$ is the sum of the idempotents corresponding to $I$.

Then, we let $A''$ be the Bongartz complement of $A$ as $\Lambda'$-module. Then  $\symm {A''}{\tau A''}$ is the \textbf{generalized Bongartz complement} of $\symm AB$ and $\symm AB\sqcup \symm  {A''}{\tau A''}$ is the generalized Bongartz completion of $\symm AB$.
\end{enumerate}
\end{Definition}
}

{
\begin{Remark}\label{rem: A'' are nonzero}
Let $\symm {A'}{B'}$ be the generalized Bongartz completion of $\symm AB$ and let $\symm {A''}{B''}=\bigsqcup\symm{A_i''}{B_i'}$ be the generalized Bongartz complement of $\symm AB$. Then, Def. \ref{def: generalized Bongartz} implies the following.
\begin{enumerate}
\item The torsion class $\,^\perp B$ is sincere and $A'$ is the Bongartz completion of $A$ if and only if none of the components of $B$ are injective.
\item The components $A_i''$ in the numerator of the generalized Bongartz complement are all nonzero.
\end{enumerate}
\end{Remark}
}

{
\begin{Theorem}\label{thm: Bongartz complement}
    A balanced pair $\symm {A}{B}$ with $k<n$ summands can be completed in a unique way to obtain a complete balanced pair 
    \[
  \symm{A'}{B'}=  \symm AB\sqcup \bigsqcup_{i=1}^{n-k} \symm{A_i}{B_i}
    \]
    where the nonzero terms in $\{A_1,\cdots,A_{n-k}\}$ are Ext-projective objects of $\,^\perp B$. More precisely, $\symm{A'}{B'}$ is the generalized Bongartz completion of $\symm AB$. Furthermore,
    \[
        \Gen\,A'=\,^\perp B=\,^\perp B'.
    \]
\end{Theorem}

\begin{proof}
The generalized Bongartz completion $\symm {A'}{B'}$ of $\symm AB$ is the unique complete balanced pair with corresponding torsion class $\Gen A'=\,^\perp B$.
\end{proof}
}

By Proposition \ref{prop: uniqueness of B+}, the lower chamber $\cB_+$ of $g(\tsymm AB)$ is characterized by the property that $\cT(\cB_+)=\,^\perp B$. This gives the following:

{
\begin{Corollary}\label{cor: vertices of U+} For 
any non-complete balanced pair $\symm{A}{B}$, the lower chamber $\cB_+$ of $g\!\left(\symm{A}{B}\right)$ exists. Furthermore, $\cB_+$ is the cone on the interior of the $(n-1)$-simplex $\Delta_+$ with vertices $g\!\left(\symm{A'_i}{B'_i}\right)$ where $\bigsqcup \symm{A'_i}{B'_i}$ is the generalized Bongartz completion of $\symm AB$.
\end{Corollary}
}


{
\begin{Definition}\label{def: coB} (From \cite{AIR}, \cite{BM})
Let $X$ be a $\tau$-rigid module. Then the collection of Ext-projective objects in the torsion class $\Gen X$ including $X$ is called the \textbf{co-Bongartz completion} of $X$ and the terms not equal to summands of $X$ form the \textbf{co-Bongartz complement} of $X$. By \cite{AIR} and \cite{BM}, the co-Bongartz completion $M$ of $X$ is support $\tau$-tilting, i.e., $(M,P)$ is $\tau$-tilting for some projective module $P$.
\end{Definition}

\begin{Definition}\label{def: generalized coB}
Let $\symm AB$ be a balanced pair. Then the \textbf{generalized co-Bongartz completion} $\symm {A'}{B'}$ of $\symm AB$ is defined as follows.
\begin{enumerate}
\item If $A=0$ then $A'=0$ and $B'$ is the sum of all indecomposable injective $\Lambda$-modules.
\item If $A\neq 0$ then $A'$ is the co-Bongartz completion of $A$, i.e., the sum of all Ext-projective objects of $\Gen A$ and $B'$ is the sum of $\tau A'$ and the indecompoable projective $\Lambda$-modules in $A^\perp$.
\end{enumerate}
In all cases $\symm {A'}{B'}=\symm AB\sqcup \symm{A''}{B''}$ and $\symm{A''}{B''}$ is called the \textbf{generalized co-Bongartz complement} of $\symm AB$.
Furthermore
\[
	\Gen A'=\Gen A=\,^\perp B'.
\]
\end{Definition}
}

{
\begin{Theorem}\label{thm: co-Bongatz}
     Any balanced pair $\symm {A}{B}$ can be completed to a unique complete balanced pair $\symm {A'}{B'}= \symm{A}{B}\sqcup \symm{A''}{B''}$ so that     
     \[
        \Gen\,A'=\Gen\,A=\,^\perp B'.
    \]
    Furthermore, $\symm {A'}{B'}$ is the generalized co-Bongartz completion of $\symm AB$.
\end{Theorem}

\begin{proof}
See \cite[Theorem 1.6(b)]{BM}.
\end{proof}
}

The upper chamber $\cB_-$ on the negative side of $g\!\left(\symm{A}{B}\right)$ is characterized by the property that $\cT(\cB_-)=\Gen\,A$ by Proposition \ref{prop: uniqueness of B-}. So, Theorem \ref{thm: co-Bongatz} gives us the following.

\begin{Corollary}\label{cor: vertices of U0} For any non-complete balanced pair $\symm{A}{B}$, the upper chamber $\cB_-$ of $g\!\left(\symm{A}{B}\right)$ exists. Furthermore $\cB_-$ is the cone on the interior of the $(n-1)$-simplex $\Delta_-$ with vertices the $g$-vectors of the components of the generalized co-Bongartz completion of $\symm AB$.
\end{Corollary}

}

{
\begin{Remark}\label{rem: almost complete A/B}
There is one important case of Corollaries \ref{cor: vertices of U+} and \ref{cor: vertices of U0} which we mentioned in the introduction. This is the case when $\symm AB$ is an almost complete balanced pair. In that case $J\left(\symm AB\right)$ has rank $1$ and has a unique simple object, say $M_0$. The $g$-vector $g\!\left(\symm AB\right)$ lies in the interior of the wall $\cD_\Lambda(M_0)$. This wall has two sides and the positive side is a red wall for $\cB_+$ and the negative side is a green wall of $\cB_-$. However, the new boundary points $g\!\left(\symm {A_0}{B_0}\right)\in\partial \cB_-$ and $g\!\left(\symm {A_0'}{B_0'}\right)\in\partial \cB_+$ might be neither red or green boundary points. (See Figure \ref{Fig: almost complete A/B}.) But we can say that $M_0$ lies in $\cT(\cB_+)=\,^\perp B$ but does not lie in $\cT(\cB_-)=\Gen A$.
\end{Remark}
}

{
We also have the following converse of Corollaries \ref{cor: vertices of U+} and \ref{cor: vertices of U0}.

\begin{Theorem}\label{thm: converse of upper/lower}
Let $\symm AB$ be a complete balanced pair where $A,B$ are both nonzero. Let $\cU=\cC\left(\symm AB\right)$ be the corresponding standard open chamber. Then there exists a unique decomposition
\[
	\symm AB=\symm{A'}{B'}\sqcup \symm{A''}{B''}
\]
where $\cU$ is the upper chamber of $g\!\left(\symm{A'}{B'}\right)$ and $\cU$ is the lower chamber of $g\!\left(\symm{A''}{B''}\right)$, i.e., $\symm AB$ is the generalized Bongartz completion of $\symm{A''}{B''}$ and the generalized co-Bongartz completion of $\symm{A'}{B'}$.
\end{Theorem}

\begin{proof}
If $A\neq0$ then $\cU=\cC\left(\symm AB\right)$ is not the negative chamber $\cU_{neg}$. If $B\neq0$ then $\cU\neq\cU_{pos}$. This means that $\cU$ has at least one red wall and at least one green wall. Let $\Delta$ be the $(n-1)$-simplex spanned by $v_i=g(\tsymm{A_i}{B_i})$. Among these vertices, let $\{v_j\mid j\in R\}$ be the set of vertices opposite the red sides of $\Delta$ (corresponding to the red walls of $\cU$) and let $\{v_k\mid k\in G\}$ be the set of vertices opposite the green sides of $\Delta$ (corresponding to the green walls of $\cU$). Let 
\[
\symm{A'}{B'}=\bigsqcup_{j\in R} \symm{A_j}{B_j},\qquad \symm{A''}{B''}=\bigsqcup_{k\in G} \symm{A_k}{B_k}
\]
Then, we claim that $\cU$ is the upper chamber of $g\!\left(\symm{A'}{B'}\right)$ and $\cU$ is also the lower chamber of $g\!\left(\symm{A''}{B''}\right)$. To see this, note that any boundary wall of $\cU$ which contains $g(\tsymm{A'}{B'})=\sum g\!\left(\tsymm{A_j'}{B_l'}\right)$ must also contain $v_j=g\!\left(\tsymm{A_j'}{B_l'}\right)$ and therefore cannot be the red wall opposite $v_j$. Since these are all the red walls, any wall containing $g(\tsymm{A'}{B'})$ must be green. So, $g(\tsymm{A'}{B'})$ is a green boundary point of $\cU$ by Remark \ref{rem: red/green walls}. So, $\cU$ is the upper chamber of $g(\tsymm{A'}{B'})$. Similarly, all walls containing $g(\tsymm{A''}{B''})$ are red. So, $g(\tsymm{A''}{B''})$ is a red boundary point of $\cU$ making $\cU$ the lower chamber of $g(\tsymm{A''}{B''})$.

With any other decomposition of $\symm AB$, either some $v_j$ will be opposite a green wall which will contain all the $v_k$ or some $v_k$ will be opposite a red wall. In the first case, $\sum v_k$ will not be a red boundary point. In the second case, $\sum v_j$ will not be a green boundary point. So, this decomposition fails to have the desired property. So, our construction gives the unique desired decomposition of $\symm AB$.
\end{proof}
}

{
\begin{Definition}\label{def: covering domain}
    Consider an indecomposable balanced pair $\symm{A_0}{B_0}$ with $A_0\neq0$. Let $\cB_-$ be the upper chamber of $g(\tsymm{A_0}{B_0})$. By Corollary \ref{cor: green path hits simple domains}, $n-1$ of the walls bounding $\cB_-$ are $\cD_\Lambda(M_i)$ for simple objects $M_i$ of $\cW(g(\tsymm{A_0}{B_0}))$. Since $A_0\neq0$, by Proposition \ref{prop: there is one positive and one negative domain}, $\cB_-$ must be on the positive side of the last wall which we call the \textbf{covering wall} and denote $\cD_\Lambda(M_-)$ as shown in Figure \ref{Fig: hour glass}.
\end{Definition}

\begin{Definition}\label{def: flooring domain}
Suppose that $B_0\neq0$ and $\cB_+$ is the lower chamber on the positive side of $g\!\left(\symm{A_0}{B_0}\right)$. Then, by Corollary \ref{cor: green path hits simple domains}, $n-1$ of the walls bounding $\cB_+$ are $\cD_\Lambda(M_i)$ for simple objects $M_i$ of $\cW(g(\tsymm{A_0}{B_0}))$. These are continuations of the same walls that form all but one of the walls of $\cB_-$. This time $\cB_+$ is on the positive side of these $n-1$ walls and, by Proposition \ref{prop: there is one positive and one negative domain}, $\cB_+$ is on the negative side of the last wall which we call the \textbf{flooring wall} and denote $\cD_\Lambda(M_+)$ as shown in Figure \ref{Fig: hour glass}.
\end{Definition}
}

\begin{Proposition}\label{prop: no map M0 to M+}
Suppose $A_0,B_0$ are both nonzero, $\cD_\Lambda(M_-)$ is the covering wall of $g\!\left(\symm{A_0}{B_0}\right)$ and $\cD_\Lambda(M_+)$ is the flooring wall of $g\!\left(\symm{A_0}{B_0}\right)$. Then $\Hom_\Lambda(M_-,M_+)=0$.
\end{Proposition}

\begin{proof}
    Let $\gamma$ be the linear green path $\gamma(t)=g(\tsymm{A_0}{B_0})+t\eta$. Then $\gamma(-\varepsilon)\in \cB_-$ and $\gamma(\varepsilon)\in \cB_+$ for small $\varepsilon>0$. By Proposition \ref{prop: there is one positive and one negative domain}, at some $t_-<0$ and $t_+>0$, $\gamma(t_-)\in\cD_\Lambda(M_-)$ and $\gamma(t_+)\in \cD_\Lambda(M_+)$. Since $t_-<t_+$ we have $\Hom_\Lambda(M_-,M_+)=0$ by Theorem \ref{thm: FHO seq}.
\end{proof}

On of the main results of this section is to give a formula for $M_-$ as a certain quotient of $A_0$ and, dually, to give a formula for $M_+$ as a submodule of $B_0$.

{Let $\symm {A_0}{B_0}$ be an indecomposable balanced pair. Let $\bigsqcup\symm{A_i}{B_i}$ be its generalized co-Bongartz complement. Let $T=\Tr_{\sqcup A_i} A_0$. Then:}

\begin{Lemma}\label{lem: trace is all}
For $i>0$ we have $\Hom_\Lambda(A_i,A_0/T)=0$.
\end{Lemma}

\begin{proof}
The short exact sequence $0\to T\to A_0\to A_0/T\to 0$ gives a long exact sequence starting with:
\[
	0\to \Hom_\Lambda(A_i,T)\xrightarrow\cong  \Hom_\Lambda(A_i,A_0)\to \Hom_\Lambda(A_i,A_0/T)\to  \Ext^1_\Lambda(A_i,T)\to \cdots.
\] 
The first map is an isomorphism since $T$ contains the trace of $A_i$ in $A_0$. Since $T$ is a quotient of a multiple of $\bigsqcup A_i$, we have $\Hom_\Lambda(T,B_i)=0$. It follows from Corollary \ref{cor: (X,B)=0 then E(A,X)=0} that $\Ext^1_\Lambda(A_i,T)=0$. Therefore $\Hom_\Lambda(A_i,A_0/T)=0$ as claimed.
\end{proof}

\begin{Lemma}\label{lem: T is not A0}
    The quotient $A_0/T$ is an iterated self-extension of $M_-$.
\end{Lemma}

\begin{proof}
Since $A_0\in \,^\perp B_i$ for all $i$, Lemma \ref{lem: trace is all} implies that $A_0/T\in A_i^\perp\cap\,^\perp B_i$ for all $i>0$. By Corollary \ref{cor: labels of domains}, the Jasso category $\bigcap (A_i^\perp\cap\,^\perp B_i)$ has rank 1 and has a unique simple object $M-$. This implies that $A_0/T$ is an iterated self-extension of $M_-$ provided it is nonzero.

    Take the linear green path $\gamma(t)=g(\tsymm{A_0}{B_0})+t\eta$ discussed in Def. \ref{def: covering domain}. Then $\gamma(t_0)\in \cD_\Lambda(M_-)$ for some $t_0<0$. Since $\gamma$ is a green path with $\gamma(t_0)\ast\undim M_-=0$, we must have
    \[
    \gamma(0)\ast\undim M_-=g\!\left(\symm{A_0}{B_0}\right)\ast \undim M_->0.
    \]
    By Corollary \ref{cor: g(A/B) dot X}, we have
    \[
       0< g\!\left(\symm{A_0}{B_0}\right)\ast \undim M_-=\dim_K \Hom_\Lambda(A_0,M_-)-\dim_K \Hom_\Lambda(M_-,B_0).
    \]
    Therefore, $\Hom_\Lambda(A_0,M_-)\neq0$. Since $M_-\in A_i^\perp\cap\,^\perp B_i$, $\Hom_\Lambda(A_i,M_-)=0$ for $i>0$ we conclude that $\Hom_\Lambda(A_0/T,M_-)\neq0$. So, $A_0/T\neq0$.
\end{proof}

{
\begin{Lemma}\label{lem: ext(A0,T)=0}
$\Ext^1_\Lambda(A_0,T)=0$ and $\Ext^1_\Lambda(A_0,M_-)=0$.
\end{Lemma}

\begin{proof}
Since $T$ is a quotient of copies of $A_-=\bigsqcup_{i>0} A_i$, $\Hom_\Lambda(T,B_0)=0$. By Corollary \ref{cor: (X,B)=0 then E(A,X)=0}, $\Ext^1_\Lambda(A_0,T)=0$.

Since $\cD_\Lambda(M_-)$ is a red wall of $\cB_-$, the point $g(\tsymm{A_-}{B_-})$ where $B_-=\bigsqcup_{i>0}B_i$ is a red boundary point of $\cB_-$. So, 
\[
\cW\left(g\!\left(\symm{A_-}{B_-}\right)\right)\subseteq \overline\cT\left(g\!\left(\symm{A_-}{B_-}\right)\right)=\cT(\cB_-)=\cT\left(\symm AB\right)
\]
Since $M_-\in \cW(g(\tsymm{A_-}{B_-}))$, $\Hom_\Lambda(M_-,B_0)=0$. By Corollary \ref{cor: (X,B)=0 then E(A,X)=0}, $\Ext^1_\Lambda(A_0,M_-)=0$.
\end{proof}

Let $T_0$ be the radical trace of $A_0$ in $A_0$. This is the sum of images of nilpotent endomorphisms of $A_0$. Equivalently, $T_0$ is the image of $\sum f_i:A_0^m\to A_0$ where $f_i$ are the generators of the unique maximal ideal $\mathfrak m$ of $\End_\Lambda(A_0)$. Let $k>1$ be minimal so that $\mathfrak m^k=0$. Then any composition of $k$ maps $f_i$ is zero. 

\begin{Lemma}\label{lem: fi not onto or into}
Let $f_1,\cdots,f_m$ be endomorphisms of a $\Lambda$-module $M$ so that any composition of $k$ of the $f_i$ is zero. Then
\begin{enumerate}
\item[(a)] $\sum f_i:M^k\to M$ is not surjective.
\item[(b)] $\bigcap \ker f_i\neq0$.
\end{enumerate}
\end{Lemma}

\begin{proof}
(a) Let $F=\sum f_i:M^m\to M$. If $F$ is surjective then so is $F^k:M^{m^k}\to\cdots\to M^{m^2}\to M^m\to M$. But $F^k=0$ since each component is a composition of $m$ of the maps $f_i$. So, $F$ is not surjective.

(b) Similarly, let $G=\sum f_i:M\to M^m$. If $G$ is injective then so is $G^k:M\to M^{m^k}$ which is zero. A contradiction. So, $G$ is a monomorphism.
\end{proof}

\begin{Proposition}\label{prop: radical trace}
$A_0/(T+T_0)\cong M_-$.
\end{Proposition}

\begin{proof}
Let $X=A_0/T$. By Lemma \ref{lem: T is not A0}, this is an iterated extension of $M_-$. First we claim that $A_0/(T+T_0)\neq0$.
\vs2
Proof: Let $\overline f_i:A_0/T\to A_0/T$ be the map induced by the nilpotent endomorphism $f_i$ of $A_0$. By Lemma \ref{lem: fi not onto or into}, $\sum \overline f_i:(A_0/T)^m\to A_0/T$ is not onto. But the image of this map is $(T+T_0)/T$. So, $T+T_0\neq A_0$.
\vs2

If $X\neq M_-$, there is a submodule $X_0\subset X$ so that $X_0\cong M_-$.  Take a map $A_0\to A_0/T$ with image $X_0$. By Lemma \ref{lem: ext(A0,T)=0}, $\Ext^1_\Lambda(A_0,T)=0$. So $A_0\to A_0/T$ lifts to $A_0\to A_0$, a nilpotent endomorphism. Repeat the process with $X/X_0$. This is an iterated extension of $M_-$. If $X/X_0\neq M_-$, there is a submodule $X_1\subset X/X_0$ so that $X_1\cong M_-$. The epimorphism $A_0\to X/X_0$ lifts to $A_0$ since, by Lemma \ref{lem: ext(A0,T)=0}, $\Ext^1_\Lambda(A_0,T\sqcup X_0)=0$. Proceeding in this way, we get a nilpotent endomorphism $f:A_0\to A_0$ so that the cokernel of $A_0\to A_0\to A_0/T$ is $A_0/(T+f(A_0))\cong M_-$. But $f(A_0)\subseteq T_0$. So, $A_0/(T+T_0)$ is a quotient of $A_0/(T+f(A_0))\cong M_-$. Since $T+T_0\neq A_0$, we must have $A_0/(T+T_0)\cong M_-$.
\end{proof}
}

Dual to the construction of $\cB_-$ and $M_-$, we have the open chamber $\cB_+$ on the positive side of $g(\tsymm{A_0}{B_0})$. The chamber $\cB_+$ has boundary consisting of $\cD_\Lambda(M_i)$, (same as for $\cB_-$) when $B_0\neq0$ we also have the flooring wall $\cD_\Lambda(M_+)$. Let $g\!\left(\symm{A_i'}{B_i'}\right)$ be the vertex of $\Delta_+$ opposite $\cD_\Lambda(M_i)$. Then
\[
    M_+\in \left(A_i'\right)^\perp\cap \,^\perp B_i'
\]
for all $i>0$.

Let $C_i$ be the \textbf{cotrace} of $B_i'$ in $B_0$. This is the intersection of all kernels of all maps $B_0\to B_i'$. Let $C=\bigcap C_i$. Analogously to Lemma \ref{lem: trace is all} we have
\[
	\Hom_\Lambda(C,B_i')=0
\]
for all $i>0$ and we conclude that $C\in A_i^\perp\cap\,^\perp B_i'$ for all $i>0$. So, $C$ is an iterated self-extension of $M_+$ if it is nonzero. Similarly to Lemma \ref{lem: T is not A0} we have:

\begin{Lemma}\label{lem: K not 0}
$C\neq0$. So, $C$ is an iterated self-extension of $M_+$.
\end{Lemma}

\begin{proof}
    Similarly to the proof of Lemma \ref{lem: T is not A0} we have $\Hom_\Lambda(M_+,B_0)\neq 0$ and any morphism $M_+\to B_0$ must have image in $C$ since $M_+\in \,^\perp B_i'$. So, $C\neq0$.
\end{proof}

{
\begin{Lemma}\label{lem: Ext(-,B0)=0}
$\Ext_\Lambda^1(B_0/C,B_0)=0$ and $\Ext_\Lambda^1(M_+,B_0)=0$.
\end{Lemma}

\begin{proof}
Since $B_0/C\subseteq \sqcup B_i^k$, $\Hom_\Lambda(A_0,B_0/C)=0$. Also, $M_+\subseteq B_0$. So, $\Hom_\Lambda(A_0,M_+)=0$. By Corollary \ref{cor: (A,Z)=0 then E(Z,B)=0} we conclude that $\Ext_\Lambda(M_+\sqcup B_0/C,B_0)=0$.
\end{proof}

Let $C_0$ be the ``radical co-trace'' of $B_0$ in $B_0$. This is $C_0=\bigcap \ker g_i$ where $g_i:B_0\to B_0$ are the generators of the maximal ideal of $\End_\Lambda(B_0)$. Then $C\cap C_0$ is the intersection of the kernels of the induced maps $\overline g_i:C\to C$. By Lemma \ref{lem: fi not onto or into}, $C\cap C_0\neq0$.

\begin{Proposition}\label{prop: C cap C0 is M+}
The canonical submodule $C\cap C_0$ of $B_0$ is isomorphic to $M_+$.
\end{Proposition}

\begin{proof}
By Lemma \ref{lem: K not 0}, $C$ is an iterated self-extension of $M_+$. Suppose $C\not\cong M_+$. Then there is an epimorphism $C\to M_+$ with kernel $Z\neq0$. Composing with the inclusion $M_+\subseteq B_0$ we get $g:C\to B_0$ with kernel $Z$. Since $\Ext_\Lambda^1(B_0/C,B_0)=0$ by Lemma \ref{lem: Ext(-,B0)=0}, $g$ extends to a map $\overline g:B_0\to B_0$ with $\ker \overline g\cap C=Z\neq0$. So, $\overline g$ is a nilpotent endomorphism of $B_0$ and $C_0\subseteq \ker \overline g$. If $Z\cong M_+$ we stop, otherwise, take an epimorphism $Z\to M_+$ with kernel $Z_0$. Since $B_0/Z$ is an extension of $M_+$ by $B_0/C$, $\Ext_\Lambda^1(B_0/Z,B_0)=0$ by Lemma \ref{lem: Ext(-,B0)=0}. So the map $Z\to M_+$ extends to $g':B_0\to B_0$ with $\ker g'\cap C=Z_0$. Proceeding in this way, we get $g'':B_0\to B_0$ with $\ker g''\cap C\cong M_+$. Since $C_0\subseteq \ker g''$ we get $C_0\cap C\cong M_+$. 
\end{proof}
}

We summarize this.

{
\begin{Theorem}\label{thm: formula for M0, M+} Let $\symm{A_0}{B_0}$ be an indecomposable balanced pair with $A_0\neq0$. Then there is an open chamber $\cB_-$ on the negative side of $g\!\left(\symm{A_0}{B_0}\right)$ which is the interior of an $(n-1)$-simplex $\Delta_-$ when intersected with the unit sphere $S^{n-1}$. All but one of the walls of $\cB_-$ are $\cD_\Lambda(M_i)$ where $M_i$ are the simple objects of the wide subcategory $\cW\left(g\!\left(\symm{A_0}{B_0}\right)\right)$. The remaining wall is the covering wall $\cD_\Lambda(M_-)$. The vertex of $\Delta_-$ opposite $\cD_\Lambda(M_i)$ is $g\!\left(\symm{A_i}{B_i}\right)$ where the nonzero $A_i$ form the co-Bongartz complement of $A_0$. Let $T$ be the trace of $\bigsqcup A_i$ in $A_0$ and let $T_0$ be the radical trace of $A_0$ in $A_0$. Then $A_0/(T+T_0)\cong M_-$. 

Dually, if $B_0\neq0$, there is an open chamber $\cB_+$ on the positive side of $g\!\left(\symm{A_0}{B_0}\right)$ with walls $\cD_\Lambda(M_i)$ and the flooring wall $\cD_\Lambda(M_+)$. The chamber $\cB_+$ is the cone on the interior of an $(n-1)$-simplex $\Delta_+$ with vertices $g\!\left(\symm{A_0}{B_0}\right)$ and $g\!\left(\symm{A_i'}{B_i'}\right)$ opposite $\cD_\Lambda(M_i)$ where the $A_i'$ form the Bongartz complement of $A_0$. Let $C$ be the intersection of all kernels of maps $B_0\to B_i'$ for $i>0$ and let $C_0$ be the intersection of all kernels of nilpotent maps $B_0\to B_0$. Then $C\cap C_0\cong M_+$.
\end{Theorem}
}

\begin{Example}\label{eg: A3 Nakayama}
Let $\Lambda_0$ be the Nakayama algebra given by a 3-cycles $1\to 2\to 3\to 1$ with $rad^7=0$. There are 9 bricks: the simples modules $S_i=P_i/rad$, the length 2 objects $P_i/rad^2$ and the length 3 objects $P_i/rad^3$. The first 6 are rigid, but the length 3 objects $P_i/rad^3$ are not rigid and these all have the same dimension vector $(1,1,1)$. Figure \ref{Fig: A3 exampls} shows the stability diagram of $\Lambda_0$. In this example, $A_0=P_2$. So, the torsion class of the chamber $\cU$ on the negative side of $v_0=g(A_0)=g(P_2)$ is $\Gen\,P_2$. We also have $v_1=g\!\left(\symm{A_1}{B_1}\right)=g\!\left(\symm{S_2}{S_3}\right)$ and $v_2=g\!\left(\symm{A_2}{B_2}\right)$ where $A_2={\genfrac{}{}{0pt}{} 23}$ and $B_2={\genfrac{}{}{0pt}{} 31}$. The only rigid objects in $\Gen\,P_2$ are $A_1=S_2$ and $A_2={\genfrac{}{}{0pt}{} 23}$, $T=S_2$ is the trace of $A_1$ and $A_2$ in $A_0$ and $T_0$, the radical trace of $A_0$ in $A_0$ is $2312$. So $A_0/T$ has composition series $231231$ making it a self-extension of $M_-=P_2/rad^3$ and $A_0/(T_0+T)=231=M_-$.
\end{Example}

\begin{figure}[htbp]
\begin{center}
\begin{tikzpicture}[scale=1.5]
\coordinate (A) at (0,1.732);
\coordinate (2A) at (0,3.464);
\begin{scope}[rotate=-60]
\begin{scope}[rotate=30]
\begin{scope}[rotate=30] 
\draw[very thick] (0,0) circle[radius=2cm];
\draw[very thick] (1,1.732) circle[radius=2cm];
\draw[very thick] (-1,1.732) circle[radius=2cm];
\begin{scope}
\clip (-2.5,-.1)rectangle (0,3.5);
\draw[very thick] (0,1.732) ellipse [x radius=2cm,y radius=1.732cm];
\end{scope}
\end{scope} 

\begin{scope}
\clip (-2,1) rectangle (2,3.1);
\draw[very thick] (0,1) ellipse [x radius=1.732cm,y radius=2cm];
\end{scope}
\end{scope}

\begin{scope}
\clip (-1,-2) rectangle (1.1,2);
\draw[very thick] (-1,0) ellipse [y radius=1.732cm,x radius=2cm];
\end{scope}
\end{scope}

\draw(-2.5,3.5) node{$\cD_\Lambda(S_1)$};
\draw(2.5,3.5) node{$\cD_\Lambda(S_2)$};
\draw(1.5,-1.5) node[right]{$\cD_\Lambda(S_3)$};
\draw(2,2.5) node{$D\binom23$};
\draw(.5,-1.2) node{$D\binom31$};
\draw(-2,2) node[left]{$D\binom12$};
\draw[black!45!green](1.85,1.5) node{$D{\tiny\begin{pmatrix} 2\\3\\1\end{pmatrix}}$};
\draw (.85,1.65) node{$v_0$};
\draw (.6,2.6) node{$v_1$};
\draw (1.55,0.95) node{$v_2$};

\draw[black!45!green](1.3,1.9) node{$\cU$};

\draw[very thick,gray] (0,1.155) circle[radius=1.78cm];

\begin{scope}
\clip (0.7,3)rectangle (-2,0.95);
\draw[very thick,blue] (0,1.155) circle[radius=1.78cm];
\end{scope}

\begin{scope}
\clip (0.7,3)rectangle (2,-.3);
\draw[very thick,black!45!green] (0,1.155) circle[radius=1.78cm];
\end{scope}
%
%
%
\end{tikzpicture}
\caption{
Here $v_0=g\!\left(\symm{P_2}0\right)$, $v_1=g\!\left(\symm{S_2}{S_3}\right)$ and $v_2=g\!\left(\tsymm{\genfrac{}{}{0pt}{} 23}{\genfrac{}{}{0pt}{} 31} \right)$. The chamber $\bg{\cU}$ has walls $\cD_\Lambda(S_1)$ on the left, $\cD_\Lambda(S_3)$ below and $D_\Lambda\left(P_2/rad^3\right)$ on the right. In the notation of Figure \ref{Fig: hour glass}, $M_-=P_2/rad^3, M_1=S_2,M_2=S_1$, $A_0=P_2,B_0=0$, $A_1=S_2$, $B_1=S_3$, $A_2=\genfrac{}{}{0pt}{} 23$ and $B_2=\genfrac{}{}{0pt}{} 31$.
The trace of $A_1$ and $A_2$ in $A_0$ is $T=S_2$, the socle of $P_2$ and $A_0/T$, with composition series $231231$, is a self extension of $M_-=P_2/rad^3$ which has composition series $231$ and $A_0/(T+T_0)\cong M_-$.
}
\label{Fig: A3 exampls}
\end{center}
\end{figure}
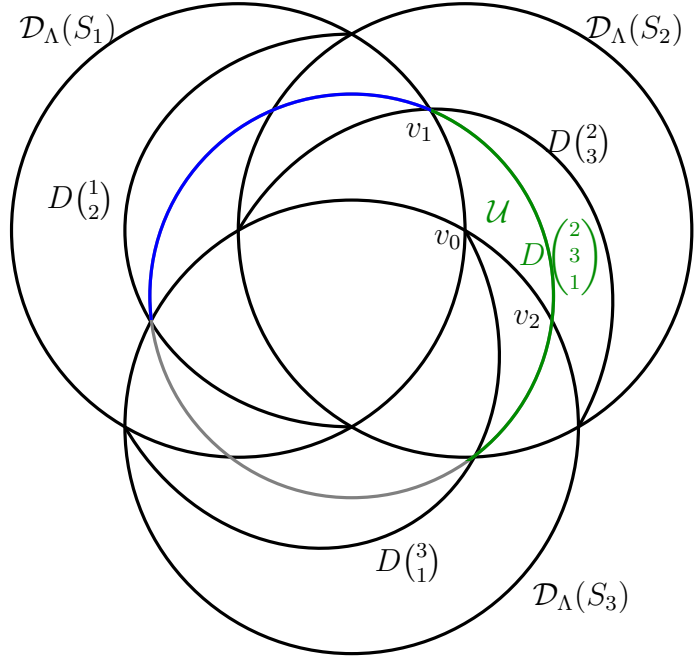
%


}



{

\section{Duality}\label{sec: duality}

There is an improved version of Theorem \ref{thm: formula for M0, M+} in the hereditary case which goes as follows.

\begin{Theorem}\label{thm: A_0=M_0, hereditary}
Given the situation in Figure \ref{Fig: hour glass} for $\Lambda=H$ a hereditary algebra, we have $M_-=A$ and $M_+=B$.
\end{Theorem}

The first statement ($M_-=A$) is proved in \cite[Proposition 2.3.33]{IOTW2}. The second statement follows by duality which we now explore.

{
Recall that duality over the ground field $K$, given by $D=\Hom_K(-,K)$ is an anti-isomorphism of categories
\[
	mod\text-\Lambda\xrightarrow{\cong} \Lambda\text-mod=mod\text-\Lambda\op.
\]

{
Duality is an exact contravariant functor with sends projective $\Lambda$-modules to injective $\Lambda\op$-modules and injective $\Lambda$-modules to projective $\Lambda\op$-modules.

We also have the following well-known statement.

\begin{Lemma}\label{lem: dual of ASS is ASS}
Let $A$ be an indecomposable nonprojective module and let $B=\tau A$. Then $DB$ is an indecomposable nonprojective $\Lambda\op$-module and $\tau DB=DA$.
\end{Lemma}

\begin{proof}
In \cite[p.144]{ARS} it says: ``It is clear that an exact sequence $0\to A\xrightarrow f B \xrightarrow g C\to 0$ is almost split if and only if $0\to DC\xrightarrow{D(g)} DB\xrightarrow{D(f)} DA\to 0$ is almost split.'' This includes the statement that $A=\tau C$ if and only if $DC=\tau DA$.
\end{proof}
}
\begin{Lemma}\label{lem: DB/DA is balanced}
If $\symm AB$ is a balanced pair in $mod\text-\Lambda$ if and only if $\symm{DB}{DA}$ is a balanced pair in $mod\text-\Lambda\op$.
\end{Lemma}
{
\begin{proof}
Given a balanced pair $\symm AB$, $A=A_0\sqcup P$ and $B=B_0\sqcup I$ where $P$ is projective, $I$ is injective satisfying the condition that $B_0=\tau A_0$ and $\Hom_\Lambda(A,B)=0$. Then $\Hom_{\Lambda\op}(DB,DA)=0$, $\tau DB_0=DA_0$, $DP$ is injective and $DI$ is projective. So,
\[
	\symm{DB}{DA}= \symm{DB_0\textstyle\sqcup DI}{DA_0\textstyle\sqcup DP}
\]is a balanced pair in $mod\text-\Lambda\op$.
\end{proof}
}

\begin{Lemma}\label{lem: DJ(A/B)}
The dual of $J\!\left(\symm AB\right)$ is $J\!\left(\symm {DB}{DA}\right)$.
\end{Lemma}

\begin{proof}
We have $X\in A^\perp\cap\,^\perp B$ if and only if $DX\in DB^\perp\cap\,^\perp DA$ since $\Hom_{\Lambda\op}(DB,DX)=\Hom_{\Lambda}(X,B)$ and $\Hom_{\Lambda\op}(DX,DA)=\Hom_{\Lambda}(A,X)$.
\end{proof}
}

{
Next we will describe how the stability diagram of $mod\text-\Lambda$ is related to that of $mod\text-\Lambda\op$.

\begin{Lemma}\label{lem: dual of DL(M)}
Let $M$ be a $\Lambda$-module with dual $DM$ and let $v\in \RR^n$. Then $v\in \cD_\Lambda(M)$ if and only if $-v\in \cD_{\Lambda\op}(DM)$.
\end{Lemma}

\begin{proof}
This follows directly from the definition of $\cD_\Lambda(M)$. Recall that $v\in \cD_\Lambda(M)$ if and only if the following hold.
\begin{enumerate}
\item $v\ast \undim M=0$ and
\item $v\ast\undim M'\le 0$ for all $M'\subseteq M$.
\end{enumerate}
For $M'\subseteq M$ with quotient $M/M'=M''$, we have a short exact sequence
\[
	0\to M'\to M\to M''\to 0.
\]
Taking duals, we reverse this to get an exact sequence
\[
	0\to DM''\to DM\to DM'\to 0.
\]
Thus, the submodules of $DM$ are the duals of the quotient modules of $M$. Furthermore, $M$ and $DM$ have the same dimension vector and:
\[
	\undim M''=\undim M-\undim M'.
\]
So, $w\in \cD_{\Lambda\op}(DM)$ if and only if the following two conditions hold.
\begin{enumerate}
\item $w\ast \undim M=0$
\item $w\ast\undim M''\le 0$ for all quotient modules $M''= M/M'$ of $M$.
\end{enumerate}
The discussion above implies that $w=-v$ is a solution of these equations. This proves the lemma.
\end{proof}
}

{
\begin{Lemma}\label{lem: Dg=-gD} $g\!\left(\symm AB\right)=-g\!\left(\symm{DB}{DA}\right)$.
\end{Lemma}

\begin{proof} We have
\[
	g\!\left(\symm AB\right)=\undim (\topp A)-\undim (\soc B)
\]	
\[
	g\!\left(\symm {DB}{DA}\right)=\undim (\topp DB)-\undim (\soc DA)
\]	
But, $\topp DB\cong D\soc B$ and $\soc DA\cong D\topp A$. So, these two vectors are negatives of each other.
\end{proof}
}

Lemma \ref{lem: dual of DL(M)} implies that $v\in\RR^n$ lies in a chamber for $\Lambda\op$ if and only $-v$ lies in a chamber for $\Lambda$. This lead to the following definition.

\begin{Definition}\label{def: DU}
    For any chamber $\cU$ of $\Lambda$, let $D\cU$ be the chamber of $\Lambda\op$ given by $D\cU=\{v\in \RR^n\mid -v\in \cU$.
\end{Definition}

\begin{Proposition}\label{prop: DB+=B-}
    The dual of the upper, resp. lower chamber of $g\!\left(\symm{A_0}{B_0}\right)$ is equal to the lower, resp. upper chamber of $g\!\left(\symm{DB_0}{DA_0}\right)$.
\end{Proposition}

\begin{proof}
    Let $v_0=g(\tsymm{A_0}{B_0})$. Then, by Lemma \ref{lem: Dg=-gD}, $-v_0=g(\tsymm{DB_0}{DA_0})$ and $\cB_-$ contains $v_0-\varepsilon\eta$ for small $\varepsilon>0$. So, $D\cB_-$ contains $-v_0+\varepsilon\eta$ which implies $D\cB_-$ is the lower chamber of $-v_0=g(\tsymm{DB_0}{DA_0})$. The proof for $D\cB_+$ is similar.
\end{proof}

{
This implies that the dual of Figure \ref{Fig: hour glass} looks like Figure \ref{Fig: Dhour glass}.

\begin{figure}[htbp]
\begin{center}
\begin{tikzpicture}
\draw[very thick] (0,0)..controls (1,1) and (3,2)..(4,2);
\draw[very thick] (0,0)..controls (-1,1) and (-3,2)..(-4,2);
\draw[very thick] (0,0)..controls (-1,-1) and (-2,-2)..(-2.5,-4);
\draw[very thick] (0,0)..controls (1,-1) and (2,-2)..(2.5,-4);
\draw[very thick] (-4,1)..controls (-3,2) and (-1.4,3)..(0,3);
\draw[very thick] (4,1)..controls (3,2) and (1.4,3)..(0,3);
\draw[very thick] (-3.5,-4)..controls (-2.5,-3.5) and (-1.5,-3).. (0,-3); 
\draw[very thick] (3.5,-4)..controls (2.5,-3.5) and (1.5,-3).. (0,-3); 
\draw[thick] (4,0)..controls (3,.5) and (1,.3)..(0,0);
\draw[thick] (4,-1)..controls (3,-.5) and (1,0)..(0,0);
\draw[thick] (4,-1.5)..controls (3,-1) and (1,-.1)..(0,0);
\draw[thick] (4,-2)..controls (3,-1.5) and (1,-.2)..(0,0);
\draw[thick] (-4,-2)..controls (-3,-1.5) and (-1,-.4)..(0,0);
%
\draw[red](0,2.5) node{\large$\cD_{\Lambda\op}(DM_+)$};
\draw[black!45!green] (0,-2.7) node{\large$\cD_{\Lambda\op}(DM_-)$};
\draw (2.3,.8) node {$\cD_{\Lambda\op}(DM_1)$};
\draw (-2.5,.8) node {$\cD_{\Lambda\op}(DM_{n-1})$};
\draw (1.7,-2.2) node[right] {$\cD_{\Lambda\op}(DM_{n-1})$};
\draw (-1.7,-2.2) node[left] {$\cD_{\Lambda\op}(DM_{1})$};
\draw (-0.1,0) node[left]{$g\!\left(\symm {DB_0}{DA_0}\right)=-v_0$};
\draw (-3.3,1.8) node[above]{$-v_1$};
\draw (3,1.5) node{$-v_{n-1}$};
\draw (2.6,-3.3) node{$-v_1'$};
\draw (-1.85,-3.65) node{$-v_{n-1}'$};
\draw (0,1.5) node{$D\mathcal B_+$};
\draw (0,-1.5) node{$D\mathcal B_-$};
\draw[thick,green!80!black,->] (-1,4)--(-1,-4);
\draw[green!80!black] (-1,4) node[left]{$\gamma$};
{
\begin{scope}[yshift=-.5mm]
\clip (-3,1.3)rectangle(3,3.1);
\draw[ thick,red] (-4,1)..controls (-3,2) and (-1.4,3)..(0,3);
\draw[thick,red] (4,1)..controls (3,2) and (1.4,3)..(0,3);
\end{scope}
\begin{scope}[yshift=.5mm]
\clip (-2.3,-4)rectangle (2.3,-2.7);
\draw[thick,black!45!green] (-3.5,-4)..controls (-2.5,-3.5) and (-1.5,-3).. (0,-3); 
\draw[thick,black!45!green] (3.5,-4)..controls (2.5,-3.5) and (1.5,-3).. (0,-3); 
\end{scope}
}
\end{tikzpicture}
\caption{Since duality negates all vectors we have the vertex $-v_0=g(\tsymm {DB_0}{DA_0})$. The walls $\cD_{\Lambda\op}(DM_i)$ pass through $-v_0$ by Lemma \ref{lem: dual of DL(M)}. The green path $\gamma$ is going in the positive direction (drawn downward). The first part of Theorem \ref{thm: A_0=M_0, hereditary} implies that, for $\Lambda=H$ hereditary, $DM_+\cong DB_0$. So, $M_+\cong B_0$.}
\label{Fig: Dhour glass}
\end{center}
\end{figure}
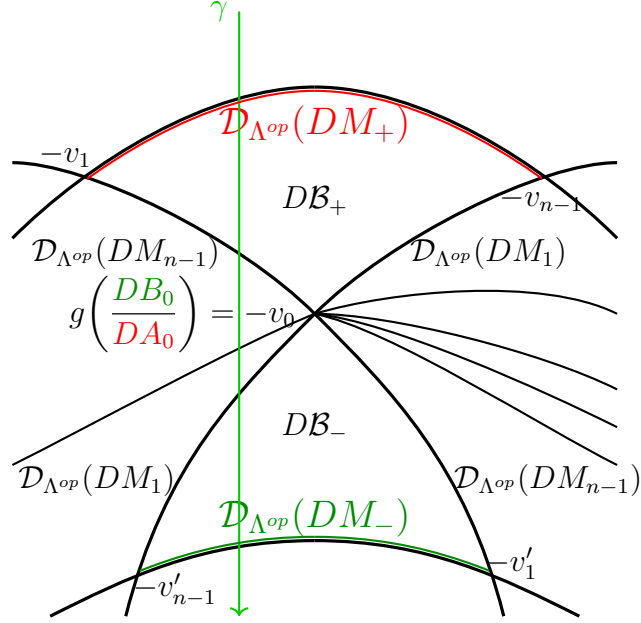
}

}



\section*{Acknowledgments} The authors thank Emre Sen for help with earlier versions of this paper. Emre solved the motivating question of this paper, due to Olivier Bernardi, whether a balanced formula for enumeration of rooted forests extended to representation of $A_n$ \cite{ISen}. Finally, the first author is grateful to the Simons Foundation for its generous support during the long process of developing this paper: Grant \#686616.
}


\begin{thebibliography}{99}

{
\bibitem{AIR} Adachi, Takahide, Osamu Iyama, and Idun Reiten. \emph{$\tau$-tilting theory}. Compositio Mathematica 150.3 (2014): 415--452.

\bibitem{ASS} Ibrahim Assem, Daniel Simpson and Andrzej Skowro\'nski. Elements of the representation theory of associative algebras: Volume 1: Techniques of representation theory. Cambridge University Press, 2006.

\bibitem{ARS} Maurice Auslander, Idun Reiten and Sverre O. Smal\o, Representation theory of Artin algebras, Cambridge Studies in Advanced Mathematics \textbf{36}, Cambridge University Press, 1995.

\bibitem{AS} Maurice Auslander and Sverre O. Smal\o, \emph{Almost split sequences in subcategories}, J. Algebra 69 (1981) 426--454. Addendum; J. Algebra 71 (1981), 592--594.

\bibitem{Br} T. Bridgeland, \emph{Scattering diagrams, Hall algebras and stability conditions}, Algebr. Geom. 4 (5)
(2017) 523--561.

\bibitem{BST}Br\"ustle, Thomas, David Smith, and Hipolito Treffinger. \emph{Wall and chamber structure for finite-dimensional algebras.} Advances in Mathematics 354 (2019): 106746. 




\bibitem{BM} A.~B.~Buan and B.~R.~Marsh, {\it $\tau$-exceptional sequences}, J. Algebra 585 (2021), 36--68.



\bibitem{DIRRT} L. Demonet, O. Iyama, N. Reading, I. Reiten, and H. Thomas, \emph{Lattice theory of torsion classes: Beyond $\tau$-tilting theory}, Trans. Amer. Math. Soc. Ser. B 10 (2023), 542--612.

\bibitem{IM} Kiyoshi Igusa and Ray Maresca, \emph{Pseudo-torsion classes}, preprint arXiv:2603.27709. 


\bibitem{IOTW2} Kiyoshi Igusa, Kent Orr, Gordana Todorov, and Jerzy Weyman, {Modulated quivers, semi-invariant pictures and picture groups}, Cambridge University Press, to appear.


\bibitem{ISen} Kiyoshi Igusa and Emre Sen, \emph{Exceptional sequences and rooted labeled forests}, Journal of Algebra and Its Applications (2024)
https://doi.org/10.1142/S0219498825501981. 





\bibitem{IngTh} Colin Ingalls and Hugh Thomas. \emph{Noncrossing partitions and representations of quivers.} Compositio Mathematica 145.6 (2009): 1533--1562.

\bibitem{Jasso} G. Jasso, \emph{Reduction of $\tau$-tilting modules and torsion pairs}, Int. Math. Res. Not. 2015 (2014), no. 16, 7190--7237.

\bibitem{K1} Bernhard Keller, \emph{On cluster theory and quantum dilogarithm identities}, Representations of algebras and related topics, EMS Ser. Congr. Rep., Eur. Math. Soc., Z\"urich, 2011, pp. 85--116.

\bibitem{K2} Bernhard Keller and Laurent Demonet. \emph{A survey on maximal green sequences}. Representation theory and beyond 758 (2020): 267--286.
}


\end{thebibliography}
\end{document}